\documentclass[12pt,reqno]{amsart}
\usepackage[a4paper, margin=1in]{geometry}
\makeatletter
\@namedef{subjclassname@2020}{%
  \textup{2020} Mathematics Subject Classification}
\makeatother
\title{Semi-integral points of bounded height on toric varieties}
\author{Alec Shute and Sam Streeter}
\address{School of Mathematics and Heilbronn Institute for Mathematical Research, University of Bristol, Fry Building, Woodland Road, BS8 1UG, UK}
\email{alec.shute@bristol.ac.uk}
\email{sam.streeter@bristol.ac.uk}
\usepackage{amsmath}
\usepackage{amssymb}
\usepackage{amsthm}
\usepackage{enumitem}
\usepackage{tikz-cd}
\usepackage{soul}
\usepackage[linktocpage=true,backref]{hyperref}
\usepackage{mathrsfs}

\theoremstyle{definition}
\newtheorem{mydef}{Definition}
\newtheorem{note}[mydef]{Note}

\newtheorem{example}[mydef]{Example}
\theoremstyle{plain}
\newtheorem{theorem}[mydef]{Theorem}
\newtheorem{proposition}[mydef]{Proposition}
\newtheorem{lemma}[mydef]{Lemma}
\newtheorem{corollary}[mydef]{Corollary}
\newtheorem{conjecture}[mydef]{Conjecture}

\numberwithin{mydef}{section}
\numberwithin{equation}{section}
\numberwithin{equation}{section}
\numberwithin{mydef}{section}
\let\originalleft\left
\let\originalright\right
\renewcommand{\left}{\mathopen{}\mathclose\bgroup\originalleft}
\renewcommand{\right}{\aftergroup\egroup\originalright}

\renewcommand{\Re}{\operatorname{Re}}

\setenumerate[0]{label=(\roman*)}

\renewcommand{\l}{\left}
\renewcommand{\r}{\right}

\newcommand{\f}{\frac}
\newcommand{\eps}{\varepsilon}

\DeclareMathOperator{\Br}{Br}

\DeclareMathOperator{\cl}{cl}

\DeclareMathOperator{\Div}{Div}

\DeclareMathOperator{\Eff}{Eff}

\DeclareMathOperator{\Gal}{Gal}
\DeclareMathOperator{\Hom}{Hom}

\DeclareMathOperator{\om}{\boldsymbol{\omega}}
\DeclareMathOperator{\Pic}{Pic}
\DeclareMathOperator{\Proj}{Proj}
\DeclareMathOperator{\rank}{rank}
\DeclareMathOperator{\red}{red}

\DeclareMathOperator{\Spec}{Spec}

\DeclareMathOperator{\vol}{vol}

\DeclareMathOperator{\PL}{PL}

\def\Z{\ifmmode{{\mathbb Z}}\else{${\mathbb Z}$}\fi}
\def\Q{\ifmmode{{\mathbb Q}}\else{${\mathbb Q}$}\fi}
\def\P{\ifmmode{{\mathbb P}}\else{${\mathbb P}$}\fi}
\def\H{\ifmmode{{\mathrm H}}\else{${\mathrm H}$}\fi}
\def\R{\ifmmode{{\mathbb R}}\else{${\mathbb R}$}\fi}
\def\F{\ifmmode{{\mathbb F}}\else{${\mathbb F}$}\fi}
\def\O{\ifmmode{{\calO}}\else{${\calO}$}\fi}
\newcommand{\A}{\mathbb{A}}

\newcommand{\m}{\mathbf{m}}

\newcommand{\calA}{{\mathcal A}}

\newcommand{\calD}{{\mathcal D}}

\newcommand{\calL}{{\mathcal L}}
\newcommand{\calM}{{\mathcal M}}

\newcommand{\calO}{{\mathcal O}}
\newcommand{\calP}{{\mathcal P}}

\newcommand{\calT}{{\mathcal T}}
\newcommand{\calU}{{\mathcal U}}
\newcommand{\calV}{{\mathcal V}}

\newcommand{\calX}{{\mathcal X}}
\newcommand{\calY}{{\mathcal Y}}
\newcommand{\calZ}{{\mathcal Z}}

\newcommand{\ZZ}{\mathbb{Z}}
\newcommand{\QQ}{\mathbb{Q}}

\newcommand{\RR}{\mathbb{R}}
\newcommand{\CC}{\mathbb{C}}
\newcommand{\FF}{\mathbb{F}}

\renewcommand{\AA}{\mathbb{A}}
\newcommand{\PP}{\mathbb{P}}

\newcommand{\sD}{\mathcal{D}}

\newcommand{\sX}{\mathcal{X}}

\newcommand{\scrB}{\mathscr{B}}

\newcommand{\scrL}{\mathscr{L}}

\DeclareMathOperator{\inv}{inv}

\renewcommand{\epsilon}{\varepsilon}

\DeclareMathOperator{\cc}{\textup{C}}
\DeclareMathOperator{\dd}{\textup{D}}
\DeclareMathOperator{\weak}{\mathbf{w}}
\DeclareMathOperator{\strong}{\mathbf{s}}
\DeclareMathOperator{\geom}{\mathbf{g}}

\DeclareMathOperator{\str}{\textup{st}}

    \DeclareFontFamily{U}{wncy}{}
    \DeclareFontShape{U}{wncy}{m}{n}{<->wncyr10}{}
    \DeclareSymbolFont{mcy}{U}{wncy}{m}{n}
    \DeclareMathSymbol{\Sh}{\mathord}{mcy}{"58} 
    \DeclareMathSymbol{\B}{\mathord}{mcy}{"42}
\begin{document}
\begin{abstract}
We prove asymptotics for semi-integral points of bounded height on toric varieties. We verify the Manin-type conjecture of Pieropan, Smeets, Tanimoto and V\'arilly-Alvarado for smooth and certain singular toric orbifolds upon replacing the leading constant with the one predicted by Chow, Loughran, Takloo-Bighash and Tanimoto.
\end{abstract}
\maketitle
\setcounter{tocdepth}{1}
\tableofcontents
\section{Introduction}

This paper concerns the intersection of two highly active areas in arithmetic geometry. One is rational points of bounded height on varieties: here we have Manin's conjecture, which predicts an asymptotic for the number of rational points of bounded height on Fano varieties. The other is semi-integral points: here we have several notions interpolating between rational and integral points, with the dual goals of better understanding integral points and of studying arithmetically special solutions to equations. Two prominent notions are \emph{Campana points}, arising from Campana's study of log-geometric orbifolds associated to fibrations \cite{C05}, and \emph{Darmon points}, originating in work of Darmon \cite{DAR} on generalised Fermat equations.

Point counting and semi-integrality were recently brought together by a conjecture of Pieropan, Smeets, Tanimoto and V\'arilly-Alvarado (Conjecture~\ref{conj:pstva}), henceforth referred to as the \emph{PSTVA conjecture}, which provides an analogue of Manin's conjecture for Campana points on log Fano orbifolds. Along with posing the conjecture, the aforementioned authors verified it for orbifolds coming from vector group compactifications \cite[Thm.~1.2]{PSTVA}. This followed earlier work of Browning, Van Valckenborgh and Yamagishi \cite{VV,BVV,BY} for linear orbifolds. Subsequent work of the authors of this article \cite{SHU1,SHU2,STR} supported the exponents of the PSTVA conjecture while raising questions about the leading constant. Further counting results for Campana points were established by Pieropan and Schindler \cite{PS} (for split toric varieties) and Xiao \cite{XIA} (for compactifications of the Heisenberg group). Chow, Loughran, Takloo-Bighash and Tanimoto \cite{CLTBT} established asymptotics for wonderful compactifications of semisimple algebraic groups and conjectured a new form for the leading constant. We henceforth refer to the PSTVA conjecture with the leading constant replaced by that of Chow, Loughran, Takloo-Bighash and Tanimoto as the \emph{modified PSTVA conjecture}.

Before presenting our results, we briefly relate other recent developments in semi-integral points. The Brauer--Manin obstruction to local-global principles for semi-integral points, pertinent to the leading constant in the modified PSTVA conjecture, was developed by Mitankin, Nakahara and the second author \cite{MNS}. Like the current formulation of Manin's conjecture, the (modified) PSTVA conjecture allows for the removal of a \emph{thin set}; the study of thin sets of Campana points was initiated by Nakahara and the second author \cite{NS}, where it was shown that Campana points on certain log Fano orbifolds are non-thin. In \cite{Moerman}, Moerman defined and studied generalisations of semi-integral points called \emph{$\calM$-points}, generalising the link with the Hilbert property established in \cite{MNS} and establishing an array of results on local-global properties for split toric varieties. Semi-integral points, particularly Darmon points, are closely connected to integral points on \emph{algebraic stacks}, for which a Manin-type conjecture was proposed by Darda and Yasuda \cite{DY} which they proved for split toric stacks \cite{DY2}. Lastly, we highlight the recent proof of Manin's conjecture for integral points on toric varieties by Tim Santens \cite{SAN}, which, alongside our work, almost completes the picture in the toric case.

\subsection{Results}

We introduce and count \emph{geometric semi-integral points} (Definition~\ref{def:geompt}), so-named as intersection multiplicity conditions are imposed relative to the geometric components of the orbifold divisor. By relating these points to ordinary semi-integral points (Corollary~\ref{cor:red}), we verify the PSTVA conjecture for smooth toric orbifolds.

We generalise Batyrev and Tschinkel's proof of Manin's conjecture for toric varieties \cite[Cor.~7.4]{BT2}, just as Pieropan et.\ al.\ \cite[Thm.~1.2]{PSTVA} generalise the result of Chambert-Loir and Tschinkel \cite[Thm.~0.1]{CLT} on vector group compactifications.

Let $T$ be a torus over a number field $K$ with splitting field $E$ and $G = \Gal(E/K)$. Let $\Sigma \subset X_*\left(\overline{T}\right)_{\mathbb{R}}$ be a complete regular polytopal $G$-invariant fan. Denote by $X_{\Sigma}$ the associated smooth projective equivariant compactification of $T$ with boundary divisor $D_{\Sigma}$. Denote by $D_{\Sigma} = \cup_{i=1}^rD_i$ the decomposition of $D_{\Sigma}$ into irreducible components over $K$. Given $\mathbf{m} = \left(m_1,\dots,m_r\right) \in \mathbb{Z}_{\geq 1}^r$, define the $\mathbb{Q}$-divisor $D_{\Sigma,\mathbf{m}} = \sum_{i=1}^r\left(1-\frac{1}{m_i}\right)D_i$.

Our first result is a Manin-type asymptotic for geometric semi-integral points.

\begin{theorem} \label{thm:geom}
For $* = \cc$ (respectively, $* = \dd$) and $S \subset \Omega_K$ finite, denote by \newline $N_{\mathbf{g}}\left(\Sigma,\calX,\m,*;B\right)$ the number of geometric $\mathcal{O}_S$-Campana points (respectively, geometric $\O_S$-Darmon points) on the orbifold $\left(X_{\Sigma},D_{\Sigma,\m}\right)$ of log-anticanonical height at most $B$ with respect to some $\mathcal{O}_{S}$-model $\mathcal{X}$ not lying on $D_{\Sigma}$. Then
\[
N_{\mathbf{g}}\left(\Sigma,\calX,\m,*;B\right) \sim c_{\mathbf{g}}\left(\Sigma,\calX,\m,*\right)B\left(\log B\right)^{\rank \Pic(X_{\Sigma}) - 1}
\]
for a constant $c_{\mathbf{g}}\left(\Sigma,\calX,\m,*\right) \in \mathbb{R}_{>0}$ as in Conjecture~\ref{conj:CLTBT} as $B \rightarrow \infty$.
\end{theorem}

Since geometric semi-integral points coincide with their non-geometric counterparts when the irreducible components of the orbifold divisor are smooth (Corollary~\ref{cor:red}), we obtain the following result, which amounts to (but is stronger than) a verification of the modified PSTVA conjecture for smooth toric orbifolds.

\begin{theorem} \label{thm:camp}
Let $\left(X_{\Sigma},D_{\Sigma,\mathbf{m}}\right)$ be as in Theorem~\ref{thm:geom}. If the divisors $D_i$ are smooth, then the modified PSTVA conjecture holds for the log-anticanonical height. In particular, the modified PSTVA conjecture holds for smooth toric orbifolds. 
\end{theorem}

As a consequence of Theorem~\ref{thm:camp}, we deduce that the modified PSTVA conjecture holds whenever $X_{\Sigma}$ is a split toric variety.

\begin{corollary} \label{cor:split}
The modified PSTVA conjecture holds for the orbifold $(X_{\Sigma},D_{\Sigma,\m})$ of Theorem~\ref{thm:geom} with log-anticanonical height whenever $X_{\Sigma}$ is a split toric variety.
\end{corollary}

Corollary~\ref{cor:split} should be compared with the main result of \cite[Thm.~1.2]{PS}, which deals with the case $K = \QQ$ and goes via the hyperbola method.

\subsection*{Plan}
In Section~\ref{sec:camp}, we define semi-integral points and state the PSTVA conjecture and its modification. In Section~\ref{sec:tor}, we give background on toric varieties. In Section~\ref{sec:har} we introduce our heights and $L$-functions. In Section~\ref{sec:fan} we introduce functions defined via the fan $\Sigma$ for the regularisation of Fourier transforms, which is the heart of the height zeta function approach. In Section~\ref{sec:geomproof} we prove our main results.


\subsection*{Conventions}
\subsubsection*{Algebra}
We denote by $R^*$ the units of a ring $R$ and by $1_G$ the identity of a group $G$. We denote by $G^\wedge = \Hom(G,S^1)$ the group of continuous characters of a topological group $G$, and by $G^\sim = \Hom(G,\mathbb{Q}/\mathbb{Z})$ the continuous $\mathbb{Q}/\mathbb{Z}$-dual, choosing an embedding $\mathbb{Q}/\mathbb{Z} \hookrightarrow S^1$ so that we may identify $G^\sim$ as a subset of $G^\wedge$. Given a perfect field $F$, we denote by $\overline{F}$ an algebraic closure of $F$ and set $G_F = \Gal\left(\overline{F}/F\right)$.
\subsubsection*{Geometry}
We write $\Spec R$ for the spectrum of a ring $R$ with the Zariski topology. An $R$-scheme is a scheme $X$ together with a morphism $X \rightarrow \Spec R$. The set of $R$-points $X\left(R\right)$ of $X$ is the set of sections of the structure morphism $X \rightarrow \Spec R$. If $X = \Proj T$ for some ring $T$ and $f \in T$, we denote by $Z\left(f\right) \subset X$ the closed subscheme $\Proj T/\left(f\right)$. Given a morphism $\Spec S \rightarrow \Spec R$, we denote by $X_S$ the fibre product $X \times_{\Spec R} \Spec S$. When $R = k$ is a field and $S = \Spec\overline{k}$, we write $\overline{X}$ for $X_{\overline{k}}$. Given $n \in \Z_{\geq 1}$, we denote by $\A^n_R$ and $\P^n_R$ the affine and projective $n$-spaces over $R$ respectively, omitting the ground ring $R$ if clear from context. A variety over a field $F$ is a geometrically integral separated scheme of finite type over $F$.
\subsubsection*{Number theory}
Given a number field $K$, we denote by $\Omega_K$ the set of places of $K$. We denote by $\Omega_K^\infty$ the archimedean places of $K$ and set $\Omega_K^f = \Omega_K \setminus \Omega_K^\infty$. For $v \in \Omega_K$, we denote by $K_v$ the completion of $K$ at $v$; if $v \not\in \Omega_K^\infty$, then $\O_v$ denotes the ring of $v$-adic integers in $K_v$, and $\pi_v$ and $\F_v$ denote a uniformiser for $\O_v$ and its residue field respectively. We set $q_v = \#\F_v$. We choose the absolute value $|\cdot|_v$ on $K_v^*$ such that $|x|_v = |N_{K_v/\mathbb{Q}_p}(x)|_p$ for $|\cdot|_p$ the usual absolute value on $\mathbb{Q}_p$. Given a finite subset $S \subset \Omega_K$ containing $\Omega_K^\infty$, we denote by $\O_S$ the ring of $S$-integers of $K$. When $S = \Omega_{K}^\infty$, write $\O_K$ for $\O_S$. We denote by $v_p$ the $p$-adic valuation for a rational prime $p \in \QQ$. Given an extension of number fields $L/K$ and $v \in \Omega_K$, we write $w \mid v$ when $w \in \Omega_L$ extends $v$. If $L/K$ is Galois with Galois group $G$, we denote by $G_v = \{g \in G: gv = v\}$ the decomposition group at $v$.
\subsubsection*{Arithmetic geometry}
Let $X$ be a variety over $K$. Let $v \in \Omega_K^f$, and let $S \subset \Omega_K$ be a finite set containing $\Omega_K^\infty$. Let $R \in \{\O_v,\O_S\}$. An \emph{$R$-model} of $X$ is a flat $R$-scheme $\calX$ of finite type together with an isomorphism between $X$ and the generic fibre of $\calX$. Suppose given an $\O_S$-scheme $\calY$, a place $v \not\in S$ and a finite set $S' \subset \Omega_K$ containing $S$. Then we denote by $\calY_{S'}$ and $\calY_v$ the base changes $\calY_{\O_{S'}}$ and $\calY_{\F_v}$ respectively. Given a $K$-variety $Z$, we will consider $Z(K)$ as a subset of $Z\left(\AA_K\right)$ via the diagonal embedding.

\subsection*{Acknowledgements}

We thank Ratko Darda, Daniel Loughran, Boaz Moerman, Marta Pieropan, Ross Paterson, Tim Santens, Damaris Schindler and Sho Tanimoto for helpful comments and suggestions. Both authors were supported by the University of Bristol and the Heilbronn Institute for Mathematical Research (HIMR). Progress was made during the workshop ``Campana Points on Toric Varieties'', held at the University of Bristol in February 2024 and funded by HIMR.

\section{Semi-integral points} \label{sec:camp}
\begin{mydef}
A \emph{Campana orbifold} over a field $F$ is a pair $\left(X,D\right)$ consisting of a proper, normal $F$-variety $X$ and an effective Cartier $\mathbb{Q}$-divisor
\[
D = \sum_{\alpha \in \mathscr{A}} \left(1-\frac{1}{m_\alpha}\right) D_\alpha
\]
on $X$ with the $D_\alpha$, $\alpha \in \calA$ irreducible and the weights $m_\alpha \in \mathbb{Z}_{\geq 1} \cup \{\infty\}$ such that only finitely many $m_\alpha \neq 1$ (by convention, we take $\frac{1}{\infty} = 0$).
The \emph{support} of the $\mathbb{Q}$-divisor $D$ is
\[
D_{\textrm{red}} = \bigcup_{m_\alpha \neq 1} D_{\alpha}.
\]
We say that $\left(X,D\right)$ is \emph{smooth} if $X$ is smooth and $D_{\textrm{red}}$ has strict normal crossings (see \cite[Def.~41.21.1,~Tag~0BI9]{SP} for the definition of strict normal crossings divisors).
\end{mydef}

\begin{example} \label{ex:orb}
To illustrate concepts, we introduce the running example $(X,D) = (\mathbb{P}^2_\mathbb{Q}, \sum_{i=0}^2(1-\frac{1}{m_i})D_i)$, where $\mathbb{P}^2$ has coordinates $x_0,x_1,x_2$ and $D_i$ is the divisor $x_i = 0$.
\end{example}

Let $\left(X,D\right)$ be a Campana orbifold over a number field $K$ and $S \subset \Omega_K$ be a finite set containing $\Omega_K^{\infty}$.

\begin{mydef}
A \emph{model} of $\left(X,D\right)$ over $\mathcal{O}_{S}$ is a pair $\left(\mathcal{X},\mathcal{D}\right)$, where $\mathcal{X}$ is a flat proper model of $X$ over $\mathcal{O}_{S}$ (i.e.\ a flat proper $\mathcal{O}_{S}$-scheme with a choice of isomorphism $\mathcal{X}_{\left(0\right)} \cong X$) and $\mathcal{D} = \sum_{\alpha \in \mathscr{A}}\l(1-\f{1}{m_{\alpha}}\r) \mathcal{D}_{\alpha}$ for $\mathcal{D}_{\alpha}$ the Zariski closure of $D_{\alpha}$ in $\mathcal{X}$.
\end{mydef}

\begin{example} \label{ex:mod}
For our running example $(\mathbb{P}^2_\mathbb{Q}, \sum_{i=0}^2(1-\frac{1}{m_i})D_i)$, we have the regular $\mathbb{Z}$-model $(\mathbb{P}^2_\mathbb{Z}, \sum_{i=0}^2(1-\frac{1}{m_i})\calD_i)$, where $\calD_i = \Proj \mathbb{Z}[x_0,x_1,x_2]/(x_i)$.
\end{example}

Now let $v \not\in S$. We write $\left(\mathcal{D}_{\red}\right)_{\O_v} = \cup_{\beta_v}\calD_{\beta_v}$ for the $\O_v$-decomposition of $\mathcal{D}_{\red} = \bigcup_{m_\alpha \neq 1} \calD_\alpha$ and $\beta_v \mid \alpha$ when $\calD_{\beta_v}$ is an $\O_v$-component of $\calD_{\alpha}$.

Let $P \in X\left(K_v\right)$, and write $\mathcal{P}_{\mathcal{O}_v} \in \mathcal{X}\left(\mathcal{O}_v\right)$ for its extension to an $\mathcal{O}_v$-point, which exists due to properness of $\mathcal{X}$. 

\begin{mydef} \label{def:intmult}
The \emph{($v$-adic) local intersection multiplicity} $n_v\left(\mathcal{Z},P\right)$ of $P \in X(K_v)$ and a closed subscheme $\mathcal{Z} \subset \mathcal{X}_{\mathcal{O}_v}$ is $\infty$ if $P \in Z = \mathcal{Z} \times_{\O_v} K_v$ and otherwise $n \in \ZZ_{\geq 0}$ such that $\mathcal{P}_{\mathcal{O}_v} \cap \mathcal{Z} \cong \Spec\left(\mathcal{O}_v/\left(\pi_v^n\right)\right)$ for $\mathcal{P}_{\mathcal{O}_v} \cap \mathcal{Z}$ the fibre product
\[
\begin{tikzcd}
\mathcal{P}_{\mathcal{O}_v} \cap \mathcal{Z} \arrow[r] \arrow[d] & \mathcal{Z} \arrow[d] \\
\Spec \O_v \arrow[r] & \mathcal{X}_{\O_v}.
\end{tikzcd}
\]
\end{mydef}

Note that this definition of intersection multiplicity coincides with the usual intersection pairing on arithmetic schemes (see e.g.\ \cite[Proof~of~Prop.~1.4.7]{VOJ}).

\begin{example} \label{ex:intmult}
In our running example, let us calculate $n_p(\calD_0,P)$ for $P = [a:b:c]$. Choosing coprime integers $a$, $b$ and $c$, we have $\calP = \Proj\mathbb{Z}[x_0,x_1,x_2]/(ax_1 - bx_0, bx_2 - cx_1, cx_0 - ax_2)$, so $\calP_{\mathbb{Z}_p} \cap (\calD_0)_{\mathbb{Z}_p} = \Proj \mathbb{Z}_p[x_1,x_2]/(ax_1,bx_2 - cx_1,-ax_2) \cong \Spec\mathbb{Z}_p/(p^{v_p(a)})$. Indeed, if $a \in \mathbb{Z}_p^*$ then the result is clear, and otherwise we have $\Proj \mathbb{Z}_p[x_1,x_2]/(ax_1,bx_2 - cx_1,-ax_2) \cong \Proj \mathbb{Z}_p[x]/(p^{v_p(a)}x)$. Then $n_p(\calD_0,P) = v_p(a)$. Similarly, $n_p(\calD_1,P) = v_p(b)$ and $n_p(\calD_2,P) = v_p(c)$.
\end{example}

Generalising the previous example, one may show that, for $Z(f)$ a hyperplane section of a projective variety $X \subset \mathbb{P}^n_K$ and $P = [x_0:x_1:\cdots:x_n] \in \mathbb{P}^n(K_v)$ with the $x_i$ a set of coprime $\mathcal{O}_v$-coordinates, we have $n_v(\calZ,P) = v(f(\mathbf{x}))$ in the standard model of $X$ coming from taking the closure of $X$ in $\mathbb{P}^n_{\O_K}$, where $\calZ$ is the closure of $Z(f)$ in this model. 

We make the following observation, which can also be found in \cite[\S2.5]{AVA} and which tells us how intersection multiplicity changes upon base change.

\begin{lemma} \label{lem:stab}
In the setting of Definition \ref{def:intmult}, let $L/K$ be a finite extension and let $w \in \Omega_L$ with $w \mid v$. Then, for $e(w/v)$ the ramification index of $L_w/K_v$,
\[
n_w(\calZ,P) = e(w/v)n_v(\calZ,P).
\]
\end{lemma}
\begin{proof}
We have the following two identities:
\begin{enumerate}
\item $\Spec A \times_{\Spec C} \Spec B \cong \Spec(A \otimes_{C}B)$ for $C$-rings $A$ and $B$.
\item $A/I \otimes_A B \cong B/IB$ for an $A$-ring $B$ and $I$ an ideal of $A$.
\end{enumerate}
From these, we deduce that
$\calP_{\O_w} \cap \calZ \cong \Spec(\O_w/(\pi_v^n))$. Then the equality follows from the identity $\pi_v = u \pi_w^{e(w/v)}$ for some $u \in \calO_w^*$ \cite[Def.~15.111.1,~Tag~0EXQ]{SP}.
\end{proof}

We will later use the following result to compute intersection multiplicities in terms of local equations for relative effective Cartier divisors (see \cite[Def.~31.18.2,~Tag~062T]{SP}).

\begin{lemma} \label{lem:affint}
Suppose that $\calZ \subset \calX_{\O_v}$ is a relative effective Cartier divisor and that the reduction $\calP_{\F_v}$ of $\calP_{\O_v}$ lies on $\calZ_{\F_v}$. Let $\Spec A$ be an affine open neighbourhood of $\calP_{\F_v}$ in $\calX$, and let $f \in A$ be a local equation for $\calZ$ at $\calP_{\F_v}$. Let $\varphi_P: A \rightarrow \O_v$ be the ring morphism corresponding to $\calP$. Then $n_v(\calZ,P) = v(\varphi_P(f))$.
\end{lemma}

\begin{proof}
Upon reducing to affines, this follows readily from the compatibility between fibre products of affine schemes and tensor products of rings: indeed,
\[
A/(f) \otimes_A \calO_v \cong \calO_v/\left(\pi_v^{v\left(\varphi_P(f)\right)}\right). \qedhere
\]
\end{proof}

\begin{mydef}
Let $P \in X\left(K_v\right)$ be a point satisfying $n_v\left(\mathcal{D}_\alpha,P\right) = 0$ for all $m_\alpha = \infty$. We say that $P$ is a \emph{$v$-adic/local:}

\begin{enumerate}
\item \emph{weak Campana point} if $\sum_{m_\alpha \neq 1,\infty }\frac{1}{m_\alpha}n_v\left(\calD_\alpha,P\right) \not\in \left(0,1\right)$;
\item \emph{Campana point} if $n_v\left(\calD_\alpha,P\right) \in \mathbb{Z}_{\geq m_\alpha} \cup \{0,\infty\}$ for all $m_\alpha \neq \infty$;
\item \emph{strong Campana point} if $n_v\left(\calD_{\beta_v},P\right) \in \mathbb{Z}_{\geq m_\alpha} \cup \{0,\infty\}$ for all $\beta_v \mid \alpha$, $m_\alpha \neq \infty$.
\item \emph{Darmon point} if $m_\alpha \mid n_v\left(\calD_{\alpha},P\right)$ for all $m_\alpha \neq \infty$;
\item \emph{strong Darmon point} if $m_\alpha \mid n_v\left(\calD_{\beta_v},P\right)$ for all $\beta_v \mid \alpha$, $m_\alpha \neq \infty$.
\end{enumerate}
We denote the sets of $v$-adic Campana points and $v$-adic Darmon points of 
$\left(\mathcal{X},\mathcal{D}\right)$
by $\left(\mathcal{X},\mathcal{D}\right)^{\cc}\left(\O_v\right)$
and $\left(\mathcal{X},\mathcal{D}\right)^{\dd}\left(\O_v\right)$
respectively.
We use the subscripts $\mathbf{w}$ and $\mathbf{s}$
to specify the weak and strong versions respectively, so that, for example, the weak $v$-adic Campana points are denoted by $\left(\mathcal{X},\mathcal{D}\right)^{\cc}_{\weak}\left(\O_v\right)$.
\end{mydef}

\begin{mydef}
We say that $P \in X(K)$ is an \emph{$\O_S$-Campana point} (or simply \emph{Campana point}) of $\left(\sX,\sD\right)$ if it is a $v$-adic Campana point for all $v \not\in S$. We make an analogous definition for the global counterpart of each of the notions of semi-integral point in the previous definition, replacing $\O_v$ by $\O_S$ in the notation.
\end{mydef}

\begin{note}
In \cite[\S7.6]{C15}, Campana points are referred to as (orbifold) integral points, while Darmon points are referred to as classical (orbifold) integral points.
\end{note}

Strong Campana points were introduced in \cite{STR} as a geometrically natural variant of Campana points which behaved well with the mildly singular orbifolds studied there. Similarly motivated by log geometry and arithmetic, we presently define a new variant of semi-integral points, which we name \emph{geometric semi-integral points}.

Write $\left(D_{\red}\right)_{\overline{K}} = \bigcup_{\gamma}D_{\gamma}$ for the decomposition of $D_{\red}$ over $\overline{K}$, and write $\gamma \mid \alpha$ to signify that $D_{\gamma}$ is a component of $D_\alpha$. Let $K_\gamma$ be the minimal field of definition of $(D_{\gamma})_{\overline{K}}$ as an irreducible component of $(D_{\red})_{\overline{K}}$ relative to the ground field $K$, the existence of which follows from \cite[Cor.~4.9.5]{EGA}. By \emph{ibid.}, the extension $K_\gamma/K$ is finite.

\begin{mydef} \label{def:geompt}
Let $P \in X\left(K_v\right)$ be a point satisfying $n_v\left(\mathcal{D}_\alpha,P\right) = 0$ for all $m_\alpha = \infty$. We say that $P$ is a \emph{$v$-adic/local:}
\begin{enumerate}
\item \emph{geometric Campana point} if $n_w(\calD_{\gamma},P) \in \mathbb{Z}_{\geq m_\alpha} \cup \{0,\infty\}$ for all $\gamma \mid \alpha$, $w \mid v \in \Omega_{K_{\gamma}}$, $m_\alpha \neq \infty$;
\item \emph{geometric Darmon point} if $m_\alpha \mid n_w(\calD_{\gamma},P)$ for all $\gamma \mid \alpha$, $w \mid v \in \Omega_{K_{\gamma}}$, $m_\alpha \neq \infty$.
\end{enumerate}
We denote by $\left(\calX,\calD\right)^{\cc}_{\mathbf{g}}(\O_v)$ and $\left(\calX,\calD\right)^{\dd}_{\mathbf{g}}(\O_v)$ the sets of local geometric Campana and Darmon points respectively and replace $\O_v$ by $\O_S$ for their global analogues.
\end{mydef}

\begin{note}
By Lemma \ref{lem:stab}, we have the following stability result: let $E/K$ be an extension over which all of the $D_{\gamma}$ are defined. By minimality of $K_\gamma$, we have $K_\gamma \subset E$. We obtain equivalent definitions for geometric Campana and Darmon points by replacing the $w$-adic multiplicities in Definition \ref{def:geompt} by the $W$-adic multiplicities for $W \mid v$ in $E$, provided that $E_W/K_{\gamma,w}$ are unramified extensions.
\end{note}

\subsection{Smoothness and semi-integral points}

In this section we note the following consequence of smoothness of the orbifold divisor on semi-integral points.

\begin{lemma}
Let $Z$ be a smooth divisor on a variety $X$ over a global field $K$. Let $\calX$ be an $\calO_S$-model of $X$ for some finite $S \subset \Omega_K$. Set $\calZ = \cl_{\calX}Z$ and let $Z_1,\dots,Z_n$ be the irreducible component of $Z_{\overline{K}}$ with $L_i/K$ the field of definition of $Z_i$. Then, there exists a finite set of places $S'\supset S$ such that, for all $v \not\in S'$, we have
\begin{enumerate}
\item For at most one $i \in \{1,\dots,n\}$, we have $n_w(\calZ_i,P) > 0$ for some $w \mid v \in \Omega_{L_i}$.
\item $n_w(\calZ_i,P) = 0$ for all $i \in \{1,\dots,n\}$ and $w \mid v \in \Omega_{L_i}$ if $Z_i$ is not defined over $K_v$.
\end{enumerate}
\end{lemma}

\begin{proof}
Since $Z$ is smooth, the $Z_i$ are disjoint; indeed, smoothness and geometric regularity coincide for schemes locally of finite type over a field \cite[Prop.~3.5.22(i)]{POO} and the irreducible components of a regular scheme are disjoint \cite[Prop.~3.5.5]{POO}. Thus, the geometric components of $(\calZ_{i})_{\FF_w}$ $w \mid v \in \Omega_{L_i}$ are disjoint for all but finitely many $v \not\in S$, thus the reduction of a local point can lie on at most one of them and can do so if and only if it is stable under the action of the decomposition group. The result follows.
\end{proof}

\begin{corollary} \label{cor:red}
Let $(X,D)$ be an orbifold over a global field $K$, where $D = \sum_{i=1}^r(1-\frac{1}{m_i})D_i$, and let $(\calX,\calD)$ be an $\calO_S$-model for some finite $S \subset \Omega_K$. If the $D_i$ are smooth (e.g.\ if $(X,D)$ is smooth), then for all but finitely many places $v \not\in S$ and $* \in \{\cc,\dd\}$, we have
\[
(\calX,\calD)^*_{\geom}(\O_v) = (\calX,\calD)^*_{\strong}(\O_v) = (\calX,\calD)^*(\O_v).
\]
\end{corollary}

\subsection{The modified PSTVA conjecture}
Having established much of the necessary notation and terminology, we will state the PSTVA conjecture and the conjecture of Chow--Loughran--Takloo--Bighash--Tanimoto for the leading constant.

Let $\left(X,D\right)$ be a smooth Campana orbifold over a number field $K$ which is \emph{klt} (i.e.\ all weights in $D$ are finite) and \emph{log Fano} (i.e.\ $-\left(K_X + D\right)$ is ample). Let $\left(\mathcal{X},\mathcal{D}\right)$ be a regular $\mathcal{O}_{S}$-model of $\left(X,D\right)$ for some finite $S \subset \Omega_K$ containing $\Omega_K^\infty$ (i.e.\ $\mathcal{X}$ is regular over $\mathcal{O}_{S}$). Let $\mathcal{L} = \left(L,\|\cdot\|\right)$ be an adelically metrised big line bundle with associated height $H_{\mathcal{L}}: X\left(K\right) \rightarrow \mathbb{R}_{>0}$ (see Definition~\ref{def:height}). For any $U \subset X\left(K\right)$ and $B \in \mathbb{R}_{>0}$, we define
\[
N\left(U,\mathcal{L},B\right) = \#\{P \in U: H_{\mathcal{L}}\left(P\right) \leq B\}.
\]

\begin{conjecture}[PSTVA conjecture] \cite[Conj.~1.1]{PSTVA} \label{conj:pstva}
Suppose that $L$ is nef and $\left(\mathcal{X},\mathcal{D}\right)\left(\mathcal{O}_{K,S}\right)$ is not thin. Then there exists a thin set $Z \subset \left(\mathcal{X},\mathcal{D}\right)\left(\mathcal{O}_{K,S}\right)$ and explicit positive constants $a = a\left(\left(X,D\right),L\right)$, $b = b\left(K,\left(X,D\right),L\right)$ and $c = c\left(K,S,\left(\mathcal{X},\mathcal{D}\right),\mathcal{L}\right)$ such that, as $B \rightarrow \infty$,
\[
N\left(\left(\mathcal{X},\mathcal{D}\right)\left(\mathcal{O}_{K,S}\right) \setminus Z,\mathcal{L},B\right) \sim c B^a\left(\log B\right)^{b-1}.
\]
\end{conjecture}

For the definition of thin sets, see \cite[\S9.1]{SER}; for the purpose of the conjecture, it is enough to regard them as ``sparse'' on varieties with abundant rational points.

\begin{note}
Note the following consequence of Corollary~\ref{cor:red}: to prove Conjecture~\ref{conj:pstva} for any orbifold with orbifold divisor having smooth irreducible components over the ground field, it suffices to prove the analogous result for geometric Campana points.
\end{note}

\subsubsection{The leading constant}

The leading constant $c = c\left(K,S,\left(\mathcal{X},\mathcal{D}\right),\mathcal{L}\right)$ is given by
\[
c\left(K,S,\left(\mathcal{X},\mathcal{D}\right),\mathcal{L}\right) = \frac{\alpha((X,D),L) \beta((X,D),L) \tau(K,S,(\calX,\calD),\calL)}{a((X,D),L) (b(K,(X,D),L)-1)!}
\]

The constant $a$ is defined by
\[
a((X,D),L) = \inf\{t \in \mathbb{R}: tL + K_X + D \in \Eff^1(X)\},
\]
where $\Eff^1(X)$ denotes the effective cone of $X$, and the constant $b$ is defined to be the codimension of the minimal supported face of $\Eff^1(X)$ containing $aL + K_X + D$.

Let us now focus on the case where $L = -K_X - D$ is the log-anticanonical divisor. We have $a = 1, b = \rank(\Pic(X))$. The definitions of $\alpha$, $\beta$ and $\tau$ of \cite[\S3.3]{PSTVA} read
\[
\begin{aligned}
\alpha((X,D),-K_X - D) & = \prod_{i=1}^r\frac{1}{m_i} \chi_{\Eff^1(X)}(-K_X - D),\\
\beta((X,D),-K_X - D) & = \#H^1(\Gamma,\Pic(\overline{X})),\\
\tau(K,S,(\calX,\calD),\calL) & = \tau_{X,D}(A),
\end{aligned}
\]
where $\tau_{X,D}$ is the Tamagawa measure as in \cite[\S3.3]{PSTVA} and $A$ is either the closure of $(\calX,\calD)^{\cc}(\O_S)$ in $(\calX,\calD)^{\cc}(\AA_{K,S})$ or the Brauer set $(\calX,\calD)^{\cc}(\AA_{K,S})^{\Br X}$ (see \cite{MNS}).

For a general line bundle $L$, the value $\chi_{\Eff^1(X)}(L) = \alpha(X,L)$ is Peyre's effective cone constant: for $(X,D) = (X_\Sigma,D_{\Sigma,\m})$ a toric orbifold, we may write $L = \sum_{i=1}^r v_i D_i$, and

\[
\chi_{\Eff^1(X)}(L) := \int_{\Eff^1(X)^*}e^{-\langle \mathbf{v}, \mathbf{y} \rangle} d\mathbf{y}. 
\]

In particular, for $L = -K_{X_{\Sigma}} - D_{\Sigma,\m}$, we have $\mathbf{v} = \mathbf{m}^{-1}$; by a change of variables $y_i \mapsto y_i/m_i$, and in view of \cite[Prop.~5.3]{BT2}, we have $\chi_{\Eff^1(X_{\Sigma})}(-K_{X_{\Sigma}} - D_{\Sigma,\m}) = \left(\prod_{i=1}^r m_i\right) \alpha(X_{\Sigma},-K_{X_{\Sigma}})$, hence $\alpha((X,D),-K_X - D) = \alpha(X,-K_X)$ in this case.

Let us now state the conjecture of Chow, Loughran, Takloo-Bighash and Tanimoto. Again, we restrict to the log-anticanonical case for simplicity.

\begin{conjecture}[Modified PSTVA conjecture] \label{conj:CLTBT} \cite[Conj.~8.3]{CLTBT}
Under the hypotheses of the PSTVA conjecture, we have, for $\calL$ a metrisation of $L=-K_X - D$,
\[
N\left(\left(\mathcal{X},\mathcal{D}\right)\left(\mathcal{O}_{K,S}\right) \setminus Z,\mathcal{L},B\right) \sim c' B^a\left(\log B\right)^{b-1},
\]
where $c' = c'(K,S,(\calX,\calD),\calL)$ decomposes as a product
\[
c' = \frac{\alpha((X,D),L)}{a((X,D),L)(b((X,D),L)-1)!} L^*(\Pic(\overline{X}),1) \lim_{\scrB} |\scrB| \lim_{S'} \tau_{X,D,S'}(\prod_{v \in S'} (\calX,\calD)^{\cc}_{\str}(\O_v)^{\scrB}),
\]
where $L^*(\Pic(\overline{X}),1)$ denotes the leading coefficient of the Artin $L$-function of $\overline{X}$ at $s = 1$, the group $\scrB$ runs over finite subgroups of $\Br(X,D) \cap \Br_1 T$, the set $S'$ ranges over finite subsets of $\Omega_K$ and the subscript $\str$ denotes points not lying on $D_{\red}$.
\end{conjecture}

For the definition of $\Br(X,D)$, see Section~\ref{sec:brauer}.

\section{Toric varieties} \label{sec:tor}

\begin{mydef}
An \emph{algebraic torus} (or simply \emph{torus}) over a field $F$ is an algebraic group $T$ over $F$ such that $T_{\overline{F}} \cong \mathbb{G}_m^n$ for some $n \in \mathbb{Z}_{\geq 1}$. The splitting field of $T$ is the Galois extension $E/F$ of minimal degree such that $T_E \cong \mathbb{G}_m^n$.
\end{mydef}

\begin{mydef}
The \emph{character group} of a torus $T$ is $X^*\left(\overline{T}\right) = \Hom\left(\overline{T},\mathbb{G}_m\right)$, and the \emph{cocharacter group} of $T$ is the dual $X_*\left(\overline{T}\right) = \Hom\left(X^*\left(\overline{T}\right),\mathbb{Z}\right)$. We set $X^*\left(T\right) = \Hom\left(T,\mathbb{G}_m\right) = X^*\left(\overline{T}\right)^{G_F}$ and $X_*\left(T\right) = \Hom\left(X^*\left(T\right),\mathbb{Z}\right) = X_*\left(\overline{T}\right)^{G_F}$.
\end{mydef}

\begin{mydef}
Let $T$ be a torus over a field $F$. We say that $T$ is:
\begin{enumerate}
\item \emph{anisotropic} if $X^*\left(T\right)$ is the trivial group, and
\item \emph{split} if $T \cong \mathbb{G}_m^n$ for some $n \in \mathbb{Z}_{\geq 1}$, i.e.\ if its splitting field is $F$.
\end{enumerate}
\end{mydef}

\begin{mydef}
A \emph{toric variety} over $F$ is a variety $X/F$ equipped with a faithful action of an algebraic torus $T$ admitting an open dense orbit containing a rational point. We call such $X$ an \emph{equivariant compactification} (or simply \emph{compactification}) of $T$.
\end{mydef}

As a first example of a toric variety, we have projective space.

\begin{example}
Note that $\mathbb{P}^n$ is a compactification of the split torus $\mathbb{G}_m^n$: we have the isomorphism $\mathbb{G}_m^n \xrightarrow{\sim} \mathbb{P}^n \setminus \bigcup_{i=0}\{x_i = 0\}$, $(a_1,\dots,a_n) \mapsto [1:a_1:\ldots:a_n]$ and the action $\mathbb{G}_m^n \times \mathbb{P}^n \rightarrow \mathbb{P}^n$, $([1:a_1:\ldots:a_n] ,[x_0:\ldots:x_n]) \mapsto [x_0:a_1x_1:\ldots:a_n x_n]$.
\end{example}

Let us now give compactifications of the two anisotropic tori in \cite[Ex.~1.1.7]{BT}.

\begin{example}
Let $L/K$ be an extension of number fields of degree $d\geq 2$ with Galois closure $E$. A choice of $K$-basis $\boldsymbol{\omega} = \{\omega_0, \ldots, \omega_{d-1}\}$ gives rise to a norm form $N_{\boldsymbol{\omega}}(x_0, \ldots, x_{d-1}) := N_{L/K}(\omega_0x_0+\cdots +\omega_{d-1}x_{d-1})$. Then the norm torus $T_{\boldsymbol{\omega}} = \PP^{d-1}\backslash Z(N_{\boldsymbol{\omega}})$ is an anisotropic torus with splitting field $E$ and compactification $X_{\om} = \PP^{d-1}_K$ with boundary divisor $D_{\om} = Z\left(N_{\om}\right)$. We have the exact sequence
\[
0 \rightarrow \mathbb{G}_m \rightarrow R_{L/K}\mathbb{G}_m \rightarrow T_{\om} \rightarrow 0,
\]
thus $T_{\om} \cong R_{L/K}\mathbb{G}_m/\mathbb{G}_m$. For $d \geq 3$, the orbifold $\left(X_{\om},\left(1 - \frac{1}{m}\right)D_{\om}\right)$ is not smooth as $D_{\om}$ is not strict normal crossings.

Asymptotics for Campana points and weak Campana points of bounded height on $\left(X_{\om},\left(1 - \frac{1}{m}\right)D_{\om}\right)$ (with respect to the obvious projective model) were established in \cite{STR} under the hypotheses that $L = E$ and that $[L:K]$ is coprime to $m$ or prime. 
\end{example}

\begin{example}
Keeping the notation from the previous example, another anisotropic torus over $K$ with splitting field $E$ is the norm-one torus $T_{\om}^1 = X_{\om}^1 \setminus H_d$, where $H_d$ is the hyperplane $x_d = 0$ and $X_{\om}^1 = \{[x_0:\dots:x_d] \in \mathbb{P}^d : N_{\om}\left(x_0,\dots,x_{d-1}\right) = x_d^d\} \subset \mathbb{P}^d$. Note that $T_{\om}^1$ fits into the exact sequence
\[
0 \rightarrow T_{\om}^1 \rightarrow R_{L/K}\mathbb{G}_m \xrightarrow{N_{L/K}} \mathbb{G}_m \rightarrow 0,
\]
thus $T_{\om}^1 \cong R^1_{L/K}\mathbb{G}_m := \ker\left(R_{L/K}\mathbb{G}_m \xrightarrow{N_{L/K}} \mathbb{G}_m\right)$.

Note that $X_{\om}^1$ is a degree-$d$ hypersurface in $\mathbb{P}^d$ which is smooth away from $D_{\om}^1$ and has singularities along the intersections of the $d$ irreducible components of $D_{\om}^1$ over $E$. When $d=3$, we obtain a cubic surface with three isolated singularities along the geometrically reducible plane section $D_{\om}^1$. In this case, one may resolve this singular locus and obtain a degree-$6$ del Pezzo surface as a smooth compactification of $R^1_{L/K}\mathbb{G}_m$ (see \cite[Thm.~A]{CT}).
\end{example}

\subsection{Fans}
We now review foundational results connecting toric varieties and fans.

Let $M$ be a free abelian group of rank $d$ with dual group $N = \Hom\left(M,\mathbb{Z}\right)$. Write $N_\mathbb{R} = N \otimes_{\mathbb{Z}}\mathbb{R}$. Denote by $\langle \cdot , \cdot \rangle : M \times N \rightarrow \mathbb{Z}$ the dual pairing and its natural $\mathbb{R}$-linear extension to a pairing $M_{\mathbb{R}} \times N_{\mathbb{R}} \rightarrow \mathbb{R}$.

\begin{mydef}
A \emph{(convex polyhedral) cone} in $N_\mathbb{R}$ is a subset of the form $\sigma = \{\sum_{i=1}^r \lambda_i v_i: \lambda_i \in \mathbb{R}_{\geq 0} \text{ for all $i$}\}$, where $\{v_1,\dots,v_r\}$ is a finite collection of vectors in $N_{\mathbb{R}}$, called \emph{generators} of $\sigma$. (We allow the generating set to be empty, in which case we have the \textit{zero cone} $\sigma = \{0\}$.) The \textit{dimension} of $\sigma$ is the dimension of the smallest linear subspace of $N_{\RR}$ containing $\sigma$. The \textit{dual cone} $\check{\sigma} \subseteq M_{\RR}$ is the cone given by 
\[
\check{\sigma} = \{u \in M_{\RR}: \langle u,v \rangle \geq 0 \textrm{ for all }v\in \sigma\}.
\]
A \textit{face} of $\sigma$ is a cone of the form 
\[
\{v \in \sigma: \langle \lambda, v\rangle = 0\}
\]
for some $\lambda \in \check{\sigma}$.
A cone is \emph{strongly convex} if it contains no line through the origin (equivalently, if $\{0\}$ is a face). A cone is \emph{rational} if it is generated by elements in $N \subset N_{\mathbb{R}}$.
\end{mydef}

\begin{mydef}
A \emph{fan} in $N_{\mathbb{R}}$ is a finite set $\Sigma$ of strongly convex polyhedral cones in $N_{\mathbb{R}}$ such that
\begin{enumerate}
\item any face of a cone in $\Sigma$ is also in $\Sigma$, and
\item the intersection of any two cones in $\Sigma$ is a face of both cones.
\end{enumerate}
The \textit{dimension} of $\Sigma$ is the maximum dimension of its cones. We say that $\Sigma$ is \emph{complete} if $N_{\mathbb{R}}$ is the union of cones from $\Sigma$ and \emph{regular} if each $\sigma \in \Sigma$ is generated by part of a $\mathbb{Z}$-basis of $N$. For $d \in \mathbb{Z}_{\geq 1}$, we denote by $\Sigma(d)$ the collection of $d$-dimensional cones in $\Sigma$.
\end{mydef}

Let $T$ be a torus over a field $F$ with cocharacter group $N = X_*\left(\overline{T}\right)$ and splitting field $E$. Set $G = \Gal(E/F)$. Note that $G$ acts on $X^*(\overline{T})$, hence on $N$. Any fan $\Sigma$ in $N_{\mathbb{R}}$ gives rise to a normal equivariant compactification $X_{\Sigma}$ over $E$. See \cite[\S1]{FUL} for details.

\begin{example}
We continue with our running example $(\mathbb{P}^2,\sum_{i=0}^2(1-\frac{1}{m_i})D_i)$ from Section~\ref{sec:camp}. Consider the split torus $T = \mathbb{G}_m^2$ over $\mathbb{Q}$. We have $X^*(\overline{T}) \cong \mathbb{Z}^2$, so $N_{\mathbb{R}} \cong \mathbb{R}^2$. Let $\{e_0,e_1\}$ be a basis for $N$ and set $e_2 := -e_0 - e_1$. Denote by $\Sigma$ the fan with $k$-dimensional cones generated by the $k$-fold subsets of $\{e_0,e_1,e_2\}$ for $k = 0,1,2$. Then $X_{\Sigma}$ is isomorphic to $\mathbb{P}^2$. The torus $T$ can be identified with the complement of the coordinate hyperplanes $D_i$, so that the $D_i$ are the irreducible components of the boundary.
\end{example}

We have the following relationships between properties of $\Sigma$ and $X_{\Sigma}$:
\begin{enumerate}
\item  $\Sigma$ is complete if and only if $X_{\Sigma}$ is proper \cite[Thm.~3.1.19(c)]{CLS}.
\item $\Sigma$ is regular if and only if $X_{\Sigma}$ is smooth \cite[Thm.~3.1.19(a)]{CLS} (cf.~\cite[Def.~3.1.18(a)]{CLS}).
\item $\Sigma$ is polytopal if and only if $X_{\Sigma}$ is projective \cite[Thm.~3.2]{BRA}. 
\item For $\Sigma$ complete, regular and $G$-invariant, $X_{\Sigma}$ is defined over $F$ \cite[Cor.,~p.~192]{VOS}.
\end{enumerate}

\begin{mydef}
Let $\Sigma \subset N_{\mathbb{R}}$ be a fan. A continuous function $\varphi: N_{\mathbb{R}} \rightarrow \mathbb{R}$ is \emph{$\Sigma$-piecewise linear} if its restriction to any cone of $\Sigma$ is linear, and is \emph{integral} if $\varphi\left(N\right) \subset \mathbb{Z}$.
\end{mydef}

Given a complete regular fan $\Sigma \subset N_{\mathbb{R}}$, denote by $e_1,\dots,e_n$ the set of primitive integral generators of the one-dimensional cones in $\Sigma$. Associated to each one-dimensional cone $\mathbb{R}_{\geq 0}e_j$ is a torus orbit $T_j \subset X_{\Sigma}$ with Zariski closure $\overline{T_j}$ (see \cite[\S4.1]{CLS} or \cite[\S1.2]{BT}).

\begin{mydef}
Given a $\Sigma$-piecewise linear function $\varphi$, define the following objects:
\begin{enumerate}
\item The divisor $D_{\varphi} = \sum_{j=1}^n \varphi\left(e_j\right)\overline{T_j}$ on $T_E$.
\item The Cartier divisor $\{U_\sigma,\left(\varphi|_{\sigma}\right)^{-1}\}_{\sigma \in \Sigma}$, where $U_{\sigma} = \Spec \overline{F}[M \cap \check{\sigma}]$
\item The invertible sheaf $\mathcal{L}\left(\varphi\right)$ associated to the above Cartier divisor.
\end{enumerate}
\end{mydef}

We will now summarise the relationship between $\Sigma$-piecewise-linear functions, divisors and invertible sheaves on $X_{\Sigma}$. For this, we introduce the following notation:
\begin{enumerate}
    \item $\PL(\Sigma)^G$ is the group of $G$-invariant $\Sigma$-piecewise-linear integral functions on $N_{\mathbb{R}}$.
    \item $\Pic^T\left(X_{\Sigma}\right)$ is the group of $T$-linearised line bundles on $X_{\Sigma}$.
    \item $\Div^T\left(X_{\Sigma}\right)$ is the group of $T$-invariant Weil divisors on $X_{\Sigma}$.
\end{enumerate}

\begin{proposition} \cite[Prop.~1.2.9,~Cor.~1.3.9]{BT}
\begin{enumerate}
\item The map $\varphi \mapsto \calL(\varphi)$ gives rise to an isomorphism between $\PL(\Sigma)^G$ and $\Pic^T\left(X_{\Sigma}\right)$.
\item The map $\varphi \mapsto D_{\varphi}$ induces an isomorphism between $\PL(\Sigma)^G$ and $\Div^T\left(X_{\Sigma}\right)$.
\end{enumerate}
\end{proposition}

\begin{mydef}
The \emph{boundary divisor} $D_{\Sigma}$ of the compactification $X_\Sigma$ of $T$ is the complement of $T$ embedded in $X_{\Sigma}$. We label its irreducible components $D_i,i=1,\dots,r$.
\end{mydef}

\begin{note} \label{note:bdry}
Note that $D_{\Sigma}$ belongs to $ -K_{X_{\Sigma}}$, the anticanonical divisor class of $X_{\Sigma}$ \cite[Thm.~8.2.3]{CLS}. The associated $\Sigma$-piecewise linear function is the function $\varphi_\Sigma$ defined by $\varphi_{\Sigma}(e_j) = 1$ for all $j =1,\dots,n$.
\end{note}

\begin{proposition} \label{prop:bij}
\begin{enumerate}[leftmargin=*]
\item \cite[Prop.~1.15(ii)]{BT2} The irreducible components $D_i$ of $D_{\Sigma}$ are in bijection with the $G$-orbits of $\Sigma(1)$.
\item \cite[\S3.1]{BT} Letting $G_i$ be the stabiliser of a primitive integral generator of a cone in $\Sigma_i(1)$, we obtain, up to isomorphism, an extension $K_i = E^{G_i}$ of $K$, with $[K_i:K]$ equal to the cardinality of $\Sigma_i(1)$. 
\end{enumerate}
\end{proposition}

By the above, we may write
\[
\Sigma(1) = \bigcup_{i=1}^r \Sigma_i(1)
\]
for the decomposition of $\Sigma(1)$ into $G$-orbits. Define $\overline{e_i}:= \sum_{e_j \in \Sigma_j(1)} e_j$.

When $T$ is split, we have $D_{\Sigma} = \cup_{j=1}^n\overline{T_j}$. For $T$ not necessarily split, we have
\[
(D_i)_E = \cup_{e_j \in \Sigma_i(1)}\overline{T}_j.
\]

Taking $v \in \Omega_K \setminus \Omega_K^\infty$, each $\Sigma_i(1)$ decomposes into a union
\[
\Sigma_i(1) = \bigcup_{\substack{w \in \Omega_{K_i} \\ w \mid v}}\Sigma_{i,w}(1)
\]
of $G_v$-orbits indexed by the places $w$ of $K_i$ over $v$, with the length of $\Sigma_{i,w}(1)$ equal to the inertia degree $f_{i,w}$ of $w$ over $v$. We have a bijection between the irreducible components of $D_i$ over $K_v$ and the $G_v$-orbits $\Sigma_{i,w}$, so that we may write
\[
\left(D_i\right)_{K_v} = \bigcup_{i,w}D_{i,w}
\]
We then have, picking any place $V$ of $E$ extending $v$ in $K_i$, the decomposition
\[
\left(D_{i,w}\right)_{E_V} = \bigcup_{e_j \in \Sigma_{i,w}(1)}\overline{T_j}.
\]

\subsection{Degree maps} \label{sec:degmaps}

Let $T$ be a torus over a number field $K$ with splitting field $E$.

\begin{mydef} \label{def:kt}
For $v \in \Omega_K$, let $T\left(\mathcal{O}_v\right)$ denote the maximal compact subgroup of $T\left(K_v\right)$. Further, set $\mathbf{K}_T:=\prod_v T(\O_v)$.
\end{mydef}
\begin{mydef}
Let $v \in \Omega_K$ and $w \in \Omega_E$ with $w \mid v$.

For $v \in \Omega_K \setminus \Omega_K^\infty$ with ramification degree $e_v$ in $E/K$, define the maps
\[
\begin{aligned}
& \deg_{T,v}: T\left(K_v\right) \rightarrow X_*\left(T_v\right), \quad t_v \mapsto [\chi_v \mapsto v\left( \chi_v\left(t_v\right)\right)], \\
& \deg_{T,E,v} = e_v \deg_{T,v}.
\end{aligned}
\]

For $v \in \Omega_K^\infty$, define the maps
\[
\begin{aligned}
& \deg_{T,v}: T\left(K_v\right) \rightarrow X_*\left(T_v\right)_{\mathbb{R}}, \quad t_v \mapsto [\chi_v \mapsto \log\left|\chi_v\left(t_v\right)\right|_v], \\
& \deg_{T,E,v} = [E_w:K_v] \deg_{T,v}.
\end{aligned}
\]

Finally, define the maps
\[
\deg_T = \sum_{v \in \Omega_K} \left(\log q_v\right) \deg_{T,v}, \quad \deg_{T,E} = \sum_{v \in \Omega_K} \left(\log q_w\right) \deg_{T,E,v}.
\]
\end{mydef}
\begin{lemma} \label{lplem} \cite[\S2.2]{BOU}, \cite[\S4.2]{LOU} Let $v \in \Omega_K$ and $f \in \{\deg_{T,v},\deg_{T,E,v}\}$.
\begin{enumerate}[label=(\roman*)]
\item \label{lplem1} If $v$ is non-archimedean, then we have the exact sequence
\[
0 \rightarrow T\left(\mathcal{O}_v\right) \rightarrow T\left(K_v\right) \xrightarrow{f} X_*\left(T_v\right).
\]
The image of $f$ is open and of finite index. Further, if $v$ is unramified in $E$, then $f$ is surjective.
 \item \label{lplem2} If $v$ is archimedean, then we have the short exact sequence
\[
0 \rightarrow T\left(\mathcal{O}_v\right) \rightarrow T\left(K_v\right) \xrightarrow{f} X_*\left(T_v\right)_{\mathbb{R}} \rightarrow 0.
\]
Further, $f$ admits a canonical section.
\item \label{lplem3} Letting $g$ be either $\deg_T$ or $\deg_{T,E}$ and denoting its kernel by $T\left(\mathbb{A}_K\right)^1$, we have the split short exact sequence
\[
0 \rightarrow T\left(\mathbb{A}_K\right)^1 \rightarrow T\left(\mathbb{A}_K\right) \xrightarrow{g} X_*\left(T\right)_{\mathbb{R}} \rightarrow 0,
\]
hence we have an isomorphism
\[
T\left(\mathbb{A}_K\right) \cong T\left(\mathbb{A}_K\right)^1 \times X_*\left(T\right)_{\mathbb{R}}.
\]
\end{enumerate}
\end{lemma}

\begin{mydef}
Let $\chi = \left(\chi_v\right)_v: T\left(\mathbb{A}_K\right) \rightarrow S^1$ be a character for a torus $T$ over a number field $K$.
\begin{enumerate}[label=(\roman*)]
\item We say that $\chi$ is \emph{automorphic} if it is trivial on $T(K)$.
\item We say that $\chi$ is \emph{unramified at $v \in \Omega_K$} if $\chi_v$ is trivial on $T\left(\mathcal{O}_v\right)$.
\item We say that $\chi$ is \emph{unramified} if it is unramified at all $v \in \Omega_K$.
\end{enumerate}
\end{mydef}

\subsection{Toric orbifolds}

Now that we have established some familiarity with toric varieties, let us discuss natural orbifolds associated to them.

\begin{mydef}
Given $\mathbf{m} = \left(m_1,\dots,m_r\right) \in \mathbb{Z}_{\geq 1}^{r}$, we define the effective Cartier $\mathbb{Q}$-divisor $D_{\Sigma,\mathbf{m}} = \sum_{i=1}^r \left(1-\frac{1}{m_i}\right)D_i$ on $X_{\Sigma}$. A \emph{toric orbifold} is an orbifold of the form $(X_{\Sigma},D_{\Sigma,\mathbf{m}})$.
\end{mydef}

\begin{example}
Our running example $(X,D) = (\mathbb{P}^2_\mathbb{Q}, \sum_{i=0}^2(1-\frac{1}{m_i})D_i)$ is a toric orbifold.
\end{example}

\begin{lemma} \cite[\S8.1,~p.~360]{CLS} \label{lem:snc}
If $X_{\Sigma}$ is smooth, then $\sum_{j=1}^n\overline{T_j}$ is a strict normal crossings divisor. In particular, all of the $\overline{T_j}$ are smooth and irreducible.
\end{lemma}

\subsection{Brauer elements and characters} \label{sec:brauer}

Let us now discuss Brauer groups associated to toric orbifolds with a view towards the conjectural leading constant of Chow et.\ al.

Recall that the \emph{Brauer group} of a $k$-variety $X$ is $\Br(X) = H^2_{\text{\'et}}(X,\mathbb{G}_m)$. For a field $k$, we set $\Br(k) := \Br(\Spec k)$. By functoriality, we have a map $\Br(k) \rightarrow \Br(X)$, the image of which we denote by $\Br_0(X)$; this map is injective when $X(k) \neq \emptyset$, thus we may identify $\Br_0(X)$ with $\Br(k)$. We also have the algebraic part $\Br_1(X) = \ker(\Br(X) \rightarrow \Br(\overline{X}))$.

In our toric setup, we have the following groups, cf.\ \cite[\S4.3]{LOU}, \cite[\S8.2]{CLTBT}:

\[
\begin{aligned}
\B(T) & = \ker (\Br_1(T) \rightarrow \prod_{v \in \Omega_K}\Br_1(T_v)), \\
\Br_e(T) & = \{\calA \in \Br(T): \calA(1_T) = 0 \in \Br(K)\}, \\
\Br(X_{\Sigma},D_{\Sigma,\mathbf{m}}) & = \{\calA \in \Br(T): m_i \partial_{D_i}(\calA) = 0 \text{ for all $i = 1,\dots,r$}\}, \\
\Br_1(X_{\Sigma},D_{\Sigma,\mathbf{m}}) & = \Br(X_{\Sigma},D_{\Sigma,\mathbf{m}}) \cap \Br_1(T), \\
\Br_e(X_{\Sigma},D_{\Sigma,\mathbf{m}}) & = \Br(X_{\Sigma},D_{\Sigma,\mathbf{m}}) \cap \Br_e(T).
\end{aligned}
\]

\begin{note}
In the definition of $\Br_e(T)$, we make implicit use of the pairing
\begin{equation} \label{eq:torbr}
\Br(T) \times T(K) \rightarrow \Br K, \quad (\calA,P) \mapsto \calA(P).
\end{equation}

This is not to be confused with the pairing
\[
\Br(T) \times T(\mathbb{A}_K) \rightarrow \mathbb{Q}/\mathbb{Z},
\]
coming from local version of \eqref{eq:torbr} at each place $v$ followed by maps $\inv_v: \Br K_v \hookrightarrow \mathbb{Q}/\mathbb{Z}$.
\end{note}

We have the following canonical isomorphism:
\[
\Br_1(T) \cong \Br_0(T) \oplus \Br_e(T).
\]

In particular, since $T(K) \neq \emptyset$, we have a canonical isomorphism
\[
\Br_e(T) \cong \Br_1(T)/\Br(K).
\]

We are now ready to state the key result linking Brauer elements and toric characters.

\begin{lemma} \label{lem:br}
Defining $(T(\AA_K)/T(K))^\wedge_{\mathbf{m}} := \{\chi \in (T(\AA_K)/T(K))^\wedge: \chi_i^{m_i} = 1 \text{ for all } i =1,\dots,r\}$,
we have an isomorphism 
\[
(T(\AA_K)/T(K))^\wedge_{\mathbf{m}} \cong  \Br_e(X_{\Sigma},D_{\Sigma,\mathbf{m}})/\B(T).
\]
\end{lemma}

\begin{proof}
First, note that $\chi \in (T(\AA_K)/T(K))^\wedge_{\mathbf{m}}$ belongs to the subgroup $(T(\AA_K)/T(K))^\sim$, as for such $\chi_i$, we have $\chi_i \in (\AA_{K_i}^*/K_i^*)^\sim$ for each $i \in 1,\dots,r$, cf.\ \cite[Proof~of~Thm.~4.9]{LOU}.

By \cite[Lem.~4.7]{LOU}, we have a commutative diagram
\[
\begin{tikzcd}
0 \arrow[r] & \B(T) \arrow[r] \arrow[d] & \Br_e(T) \arrow[r] \arrow[d] & \left(T(\AA_K)/T(K)\right)^\sim \arrow[d,"\gamma"] \arrow[r] & 0\, \\
 & 0 \arrow[r] & \bigoplus_{i=1}^r \Br_e\left(R_{K_i/K}\mathbb{G}_m/\mathbb{G}_m\right) \arrow[r] & \bigoplus_{i=1}^r \left(\AA_{K_i}^*/K_i^*\right)^\sim \arrow[r] & 0, 
\end{tikzcd}
\]
with exact rows, and the diagram
\[
\begin{tikzcd}
\Br_e T \arrow[r] \arrow[d] & \bigoplus_{i=1}^r \Br_e\left(R_{K_i/K}\mathbb{G}_m/\mathbb{G}_m\right)  \arrow[d, "\sim"] \\
\Br_1 T \arrow[r,"\oplus_{i=1}^r\partial_{D_i}"] & \bigoplus_{i=1}^rH^1(K_i,\QQ/\ZZ),
\end{tikzcd}
\]
which commutes up to sign. In particular, to each $\chi \in (T(\AA_K)/T(K))^\sim$ we may associate an element $\calA_\chi \in \Br_e T$, well-defined up to addition of elements in $\B(T)$. Using the second diagram, we see that $m_i \partial_{D_i}(\calA_\chi) = 0$ iff $\chi \in (T(\AA_K)/T(K))^\wedge_{\m}$, hence the result.
\end{proof}

\section{Harmonic analysis} \label{sec:har}

\subsection{Heights} \label{sec:heights}

Let $X$ be a variety over a number field $K$, and let $\scrL$ be a line bundle on $X$. We make the following definitions.

\begin{mydef}
Given $v \in \Omega_K$, a \emph{$v$-adic metric} $\|\cdot\|_v$ on $\scrL$ is a family of norms on the stalks $\scrL_x$, $x \in X(K_v)$, varying continuously for the $v$-adic topology on $X(K_v)$. 
\end{mydef}

Of particular pertinence to the study of semi-integral points are \emph{model metrics}, which are defined for varieties over a number field as follows.

\begin{mydef} \cite[\S2.1.5]{CLT2}
Let $\calX$ be a flat proper $\O_S$-model of $X$ for some finite set of places $S \subset \Omega_K$, and let $\overline{\scrL}$ be an extension of $\scrL$ to a line bundle on $\calX$. For $x \in X(K_v)$, $v \not\in S$, denote by $\overline{x}: \Spec \O_v \rightarrow \calX_{\O_v}$ its unique extension to an $\O_v$-point by properness. Define the \emph{model metric} $\|\cdot\|_{v,x}$ on $\scrL_x \otimes K_v$ associated to the pair $(\calX,\overline{\scrL})$ by setting $\{s: \|s\|_{v,x} \leq 1\} = \overline{x}^* \overline{\scrL}$, which is a lattice in $\scrL_x \otimes K_v$. 
\end{mydef}

Model metrics are closely connected to intersection multiplicities.

\begin{lemma} \label{lem:BG} \cite[Example~2.7.20]{BG}
Let $\scrL = \calO_X(D)$ be the line bundle corresponding to an effective Cartier divisor $D \subset X$, and let $s_D$ be the canonical section of $\scrL$ cutting out $D$. Then, for all $v \not\in S$ and $\|\cdot\|_v$ the model metric on $\scrL$, we have, for all $P \in (X \setminus D_{\text{red}})(K_v)$, the equality
\[
n_v(\calD,P) = \log_{q_v}\|s_D(P)\|_{v}^{-1}.
\]
\end{lemma}

\begin{mydef} 
\begin{enumerate}[leftmargin=*]
    \item An \emph{adelic metric} $\|\cdot\| = \left(\|\cdot\|_v\right)$ on $\scrL$ is a collection of $v$-adic metrics of $\scrL$ for each $v \in \Omega_K$ which coincide with the model metric relative to some fixed pair $(\calX,\overline{\scrL})$ at all but finitely many places.
    \item An \emph{adelically metrised line bundle} $\calL$ on $X$ is a pair $\left(\scrL,\|\cdot\|\right)$ of a line bundle $\scrL$ on $X$ and an adelic metric $\|\cdot\|$ on $\scrL$.
\end{enumerate}
\end{mydef}

\begin{mydef} \label{def:height}
Let $s \in \Gamma(X,\scrL) \setminus \{0\}$ and $\calL$ be as above. We define the height $H_{\calL,s}$ by
\[
H_{\calL,s}: X(\mathbb{A}_K) \rightarrow \mathbb{R}_{>0} \cup \{\infty\}, \quad H_{\calL,s}\left(\left(x_v\right)_v\right) = \prod_v  \|s(x_v)\|_v^{-1}.
\]
The restriction of $H_{\calL,s}$ to $X(K)$ is independent of $s$ by the product formula; we thus denote it by $H_{\calL}: X(K) \rightarrow \mathbb{R}_{>0}$.
\end{mydef}

\begin{mydef}[Batyrev--Tschinkel height]
Let $\varphi \in \PL\left(\Sigma\right)^G_{\mathbb{C}}$. Given $t_v \in T\left(K_v\right)$, set $\overline{t_v}:=\deg_{T,E,v}(t_v) \in X_*\left(T_v\right)_{\mathbb{R}}$, and denote by $\langle \cdot,\cdot\rangle: \PL\left(\Sigma\right)^G_{\mathbb{C}} \times X_*(T_v)_{\RR}$ the pairing coming from the degree map. Set $q_v :=e$ for $v \mid \infty$. Then we define $H_{\Sigma,v}(\varphi,\cdot)$ by
\[
H_{\Sigma,v}\left(\varphi,t_v\right) = q_v^{\langle \varphi,\overline{t_v}\rangle}.
\]
We then obtain a height on adelic points
\[
H_{\Sigma}(\varphi,\cdot): T\left(\mathbb{A}_K\right) \rightarrow \mathbb{R}_{>0}, \quad H_{\Sigma}\left(\varphi,\left(t_v\right)_v\right) = \prod_v H_{\Sigma,v}\left(\varphi,t_v\right),
\]
which becomes a height on $T(K)$ via the diagonal embedding $T(K) \hookrightarrow T(\A_K)$.
\end{mydef}

\begin{note}
For $\varphi_\Sigma$ as in Note \ref{note:bdry}, we have $H_\Sigma(\varphi_\Sigma,\cdot) = H_\calL$ for $\calL$ some adelic metrisation of the anticanonical bundle on $X_\Sigma$, cf.\ \cite[Rem.~2.1.8]{BT}.

Further, the proof of \cite[Thm.~2.1.6(iv)]{BT} requires only nefness of $D_{\varphi}$ to identify $H_{\Sigma,v}$ at all but finitely many places with the local Weil function for $D_{\varphi}$, thus one sees that this metrisation is then given by intersection with $D_{\varphi}$; in particular, when $-K_{X_\Sigma}$ is very ample, the adelic metrisation comes from intersection with $D_{\Sigma}$.
\end{note}

\begin{proposition} \label{prop:cochar}
Let $X_{\Sigma}$ be a compactification of a torus $T$ over a number field $K$ with splitting field $E$. Let $\calX$ be a flat proper $\O_S$-model of $X_\Sigma$ for some finite $S \in \Omega_K$. Let $\overline{\calT_j}$ be the closure of $\overline{T_j}$ in $\calX_{\O_{S_E}}$ for each $j \in \{1,\dots,n\}$, where $S_E = \{w \in \Omega_E: \exists v \in S \text{ s.t. } w \mid v, \}$. Suppose that the fan $\Sigma$ is complete and regular. Suppose that $\overline{t_v} \in \sigma$ for some cone $\sigma \in \Sigma$, and write $\overline{t_v} = \sum_{\langle e_j \rangle \in \sigma(1)}\lambda_j e_j$ for some $\lambda_j \in \mathbb{Z}_{\geq 0}$. Then there exists a finite set of places $S(\Sigma,\calX) \supset \Omega_K^\infty$ such that, for $v \not\in S(\Sigma,\calX)$, we have $\lambda_j = e(w/v) n_w(\overline{\calT_j},t_v)$ for any $w \in \Omega_E$ with $w \mid v$ with ramification degree $e(w/v)$.
\end{proposition}

\begin{proof}
By Lemma~\ref{lem:stab}, we can and do reduce to the case where $K = E$ (i.e.\ $T$ is split).

Let $\{e_1,\dots,e_n\}$ be a set of primitive integral generators for the one-dimensional cones of $\Sigma$. Since $\Sigma$ is complete, $\overline{t_v}$ belongs to a cone $\sigma \in \Sigma$; since $\Sigma$ is regular, $\sigma$ is generated by a subset of the $e_j$, so $\overline{t_v} = \sum_{\langle e_j \rangle \in \sigma(1)} \lambda_j e_j$ for some $\lambda_j \in \mathbb{Z}_{\geq 0}$, as claimed implicitly.

For each $e_k$, we have $\varphi_k \in \PL(\Sigma)$ defined by $\varphi_k(e_j) = \delta_{kj}$. Then $\varphi_k(\overline{t_v}) = \lambda_k$. Then it suffices to show that $H_{\Sigma}(\varphi_k,t_v) = q_v^{n_v(\overline{\calT_k},t_v)}$ for all but finitely many places $v$. This follows from the fact that the Batyrev--Tschinkel height arises from an adelic metric (see \cite[Cor.~3.19]{BOU}) and Lemma~\ref{lem:BG} (see also \cite[Prop.~and~Def.~9.2]{SAL}). 
\end{proof}

\begin{mydef} \label{def:badplaces}
Define $S(\Sigma,\calX) \subset \Omega_K$ to be the set of places $v$ as in the proof of Proposition~\ref{prop:cochar}. Denote by $S'(\Sigma,\calX) \supset S(\Sigma,\calX)$ the union of $S(\Sigma,\calX)$ with the set of places of $K$ ramified in $E/K$.
\end{mydef}

\begin{corollary} \label{cor:maincor}
Let $v \not\in S'(\Sigma,\calX)$. Writing $\overline{t_v} = \sum_{\overline{e}_{i,w} \in \sigma(1)} \alpha_{i,w}\overline{e}_{i,w}$ with $\alpha_{i,w} \in \mathbb{Z}_{\geq 0}$, the point $t_v$ is:
\begin{enumerate}
\item A local geometric Campana point (respectively, a local geometric Darmon point) if and only if $\alpha_{i,w} \in \mathbb{Z}_{\geq m_i} \cup \{0\}$ (respectively, $m_i \mid \alpha_{i,w}$).
\item A local strong Campana point (respectively, a local strong Darmon point) if and only if $\alpha_{i,w} \in \mathbb{Z}_{\geq \frac{m_i}{f_{i,w}}} \cup \{0\}$ (respectively, $m_i \mid f_{i,w} \alpha_{i,w}$);
\end{enumerate}
\end{corollary}

\begin{proof}
The result follows from Proposition~\ref{prop:cochar}, Lemma~\ref{lem:stab} and additivity of intersection multiplicity on components.
\end{proof}

\begin{note}
Note that $\varphi \in \PL(\Sigma)_\mathbb{C}$ is $G$-invariant if and only if $\varphi(e_{j_1}) = \varphi(e_{j_2})$ for all $e_{j_1}, e_{j_2} \in \Sigma_i(1)$, $i =1,\dots,r$. Then $\varphi \in \PL(\Sigma)_\mathbb{C}^{G}$ is determined by $\varphi(\Sigma_i(1))$, $i \in \{1,\dots,r\}$.
\end{note}

Just as it proved fruitful to work with the anticanonical height on toric varieties in \cite{BT2}, it will prove natural and fruitful for us to work with the \emph{log-anticanonical height}.

\begin{mydef}
The \emph{log-anticanonical height} for $(X_\Sigma,D_{\Sigma,\mathbf{m}})$ is the height corresponding to the $\Sigma$-piecewise linear function $\varphi_{\Sigma,\mathbf{m}}$ defined by $\varphi_{\Sigma,\mathbf{m}}(\Sigma_i(1)) = \frac{1}{m_i}$. We denote this height by $H_{\m}: T(\AA_K) \rightarrow \mathbb{R}_{>0}$, and we set
\[
H_{\m}(\mathbf{s},t):=H_\Sigma(\varphi_{\Sigma,\m} \cdot \varphi_{\mathbf{s}},t)
\]
for $\mathbf{s} \in \mathbb{C}^r$, where $\varphi_{\mathbf{s}}(\Sigma_i(1)) = s_i$.
We write $H_{\m,v}$ for the associated local height on $T(K_v)$.
\end{mydef}

\begin{mydef}
Let $v \in \Omega_K$, and let $\left(\calX,\calD\right)$ be an $\O_S$-model for $\left(X_{\Sigma},D_{\Sigma,\m}\right)$.

Define the functions $\delta_{\m,v}^*,\delta_{\mathbf{m},\geom,v}^*: X_{\Sigma}\left(K_v\right) \rightarrow \{0,1\}$ as follows:
\begin{enumerate}[label=(\roman*)]
    \item For $v \not\in S$, let $\delta_{\m,v}$ (respectively, $\delta_{\mathbf{m},\geom,v}^*$) be the indicator function for $\left(\calX,\calD\right)^*\left(\O_v\right)$ (respectively, $\left(\calX,\calD\right)_{\geom}^*\left(\O_v\right)$).
    \item For $v \in S$, let $\delta_{\m,v}^*$ and $\delta_{\mathbf{m},\geom,v}^*$ be identically $1$.
\end{enumerate}
We then define the indicator functions
\[
\begin{aligned}
\delta_{\mathbf{m}}^*:& X_{\Sigma}\left(\mathbb{A}_K\right) \rightarrow \{0,1\}, \quad \left(x_v\right)_v \mapsto \prod_{v \in \Omega_K}\delta_{\mathbf{m},v}^*\left(x_v\right), \\
\delta_{\mathbf{m},\geom}^*:& X_{\Sigma}\left(\mathbb{A}_K\right) \rightarrow \{0,1\}, \quad \left(x_v\right)_v \mapsto \prod_{v \in \Omega_K}\delta_{\mathbf{m},\geom,v}^*\left(x_v\right),
\end{aligned}
\]

and we define $\delta_{\m}^*$ and $\delta_{\mathbf{m},\geom}^*$ on $X_{\Sigma}\left(K\right)$ via the diagonal embedding $X_{\Sigma}(K) \hookrightarrow X_{\Sigma}\left(\AA_K\right)$.
\end{mydef}

\begin{mydef}
For $\Re(\mathbf{s}) > 1$ and $\varphi \in \PL(\Sigma)^G_{\CC}$, define the functions
\[
Z_{\mathbf{m}}^*(\mathbf{s}) = \sum_{P \in T(K)} \frac{\delta_{\mathbf{m}}^*\left(P\right)}{H_{\m}(\mathbf{s},P)}, \quad Z_{\mathbf{m},\geom}^*(\mathbf{s}) = \sum_{P \in T(K)} \frac{\delta_{\mathbf{m},\geom}^*\left(P\right)}{H_{\m}(\mathbf{s},P)}.
\]
\end{mydef}

\begin{mydef}
Let $\mu_v$, $v \in \Omega_K$ and $\mu$ be the Haar measures on $T\left(K_v\right)$ and $T\left(\mathbb{A}_K\right)$ respectively as defined in \cite[\S3]{ONO}.
\end{mydef}

\begin{mydef}
Let $\chi$ be a character of $T\left(\mathbb{A}_K\right)$. Let $\delta = \prod_v\left(\delta_v\right)_v$ be a function on $T(\A_K)$, meaning that $\delta_v(T(\O_v)) = 1$ for all but finitely many $v \in \Omega_K$, and let $\mathbf{s} \in \CC^r$. We define, for each $v \in \Omega_K$, the \emph{$v$-adic/local Fourier transform of $\chi$ with respect to $\delta$} by
\[
\widehat{H}_{\m,v}\left(\delta_v,\chi_v;-\mathbf{s}\right) = \int_{T\left(K_v\right)} \frac{\delta_v\left(t_v\right)\chi_v\left(t_v\right)}{H_{\m,v}\left(\mathbf{s},t_v\right)}d\mu_v.
\]
We then define the \emph{global Fourier transform of $\chi$ with respect to $\delta$} by
\[
\widehat{H}_{\m}\left(\delta,\chi_v;-\mathbf{s}\right) = \int_{T\left(\mathbb{A}_K\right)} \frac{\delta\left(t\right)\chi\left(t\right)}{H_{\m,v}\left(\mathbf{s},t\right)}d\mu = \prod_v \int_{T\left(K_v\right)} \frac{\delta_v\left(t_v\right)\chi_v\left(t_v\right)}{H_{\m,v}\left(\mathbf{s},t_v\right)}d\mu_v.
\]
\end{mydef}

We note the following important result which simplifies our analysis.

\begin{proposition}
For all $v \in \Omega_K$, there exists a compact open subgroup $\mathbf{K}_{\m,v}$ (respectively, $\mathbf{K}_{\m,\geom,v}^* \subset T(\mathcal{O}_v)$ of finite index such that $\delta_{\mathbf{m},v}^*$ (respectively, $\delta_{\mathbf{m},\geom,v}^*$) and $H_{\Sigma,v}(\varphi,\cdot)$ are invariant and $1$ on $\mathbf{K}_{\m,v}^*$ (respectively, $\mathbf{K}_{\m,\geom,v}^*$) for all $v \in \Omega_K$ and all $\varphi \in \PL(\Sigma)^G_{\mathbb{C}}$. Moreover, $\mathbf{K}_{\m,v}^* = \mathbf{K}^*_{\m,\geom,v} = T(\calO_v)$ for $v \not\in S(\Sigma,\calX)$.
\end{proposition}

\begin{proof}
$H_{\Sigma,v}(\varphi,\cdot)$ is $T(\O_v)$-invariant for all $v \in \Omega_K$, cf.\ \cite[Thm.~2.1.6]{BT}, so we focus on $\delta_v \in \{\delta_{\m,v}^*,\delta_{\m,\geom,v}^*\}$. The proof of existence of $\mathbf{K}_{\m,\geom,v}^*$ is analogous to the vector group case handled in \cite[Lem.~3.2]{CLT}. The key ingredients are that $\delta_v$ is locally constant (since the reduction map is continuous) and that $T(K_v)$ is a locally compact totally disconnected group, so that an open neighbourhood of $1$ contains a compact open subgroup. That $\mathbf{K}_{\m,v}^* = \mathbf{K}_{\m,\geom,v}^* = T(\O_v)$ for $v \not\in S(\Sigma,\calX)$ follows from Proposition~\ref{prop:cochar}.
\end{proof}

\begin{mydef}
Set $\mathbf{K}^*_{\m}:=\prod_{v \in \Omega_K} \mathbf{K}_{\m,v}^*$ and $\mathbf{K}^*_{\m,\geom}:=\prod_{v \in \Omega_K} \mathbf{K}_{\m,\geom,v}^*$.
\end{mydef}

\begin{corollary} \label{cor:nr}
If $\chi\left(\mathbf{K}^*_{\m}\right) \neq 1$ (respectively, $\chi\left(\mathbf{K}^*_{\m,\geom}\right) \neq 1$), then $\widehat{H}_{\m}(\delta_{\m}^*, \chi; -\mathbf{s})=0$ (respectively, $\widehat{H}_{\m}(\delta_{\m,\geom}^*, \chi; -\mathbf{s})=0$).
\end{corollary}

\begin{proof}
Suppose that $\chi_v\left(\mathbf{K}^*_{\m,v}\right) \neq 1$. The functions $\delta_{\m,v}^*$ and $H_{\m,v}(\varphi_{\mathbf{s}},\cdot)$ are $\mathbf{K}_{\m,v}^*$-invariant; interpreting them on  $\mathbf{K}_{\m,v}^* \subset X_*(T_v)$, we have
\begin{equation}
\widehat{H}_{\m,v}(\delta_{\m,v}^*, \chi_v;-\mathbf{s}) = \sum_{n_v \in T(K_v)/\mathbf{K}_{\m,v}^*}\delta_{\m,v}^*(n_v)H_{\m,v}(-\mathbf{s},n_v)\int_{\mathbf{K}_{\m,v}^*}\chi_v(t_v)\textrm{d}\mu_v,
\end{equation}
which is zero by character orthogonality. The geometric case is analogous.
\end{proof}

\begin{proposition} \label{prop:lftsarenice}
Let $v\in \Omega_K$ and let $\chi_v$ be a character of $T(K_v)$. Then for any $\eps>0$ and $\delta_v \in \{\delta_{\m,v}^*,\delta_{\m,\geom,v}^*\}$, the local Fourier transform $\widehat{H}_{\m,v}(\delta_v, \chi_v;-\mathbf{s})$ is absolutely convergent and uniformly bounded in the region $\Re(\mathbf{s})\geq \eps$.
\end{proposition}

\begin{proof}
We have
\begin{equation}
|\widehat{H}_{\m,v}(\delta_v, \chi_v;-\mathbf{s})| \leq \int_{T(K_v)}\l|\f{\delta_v(t_v)\chi_v(t_v)}{H_{\m,v}(\mathbf{s},t_v)}\r|\textrm{d}\mu_v \leq \widehat{H}_{\m,v}(1,1;-\eps). 
\end{equation}
Uniform boundedness of $\widehat{H}_{\m,v}(1,1;-\eps)$ follows from \cite[Rem.~2.2.8,~Prop.~2.3.2]{BT}.
\end{proof}

\begin{proposition} \label{prop:nonvan}
For any $v \in \Omega_K$ and $\delta_v \in \{\delta_{\m,v}^*,\delta_{\m,\geom,v}^*\}$, the local Fourier transform $\widehat{H}_{\m,v}(\delta_v, 1;-\mathbf{s})$ is non-vanishing for $\mathbf{s} \in \R_{>0}$. 
\end{proposition} 

\begin{proof}
Recall that $\delta_{\m,v}^*(\mathbf{K}_{\m,v}^*) = \delta_{\m,\geom,v}^*(\mathbf{K}_{\m,\geom,v}^*) = 1$. Since $T(K_v)$ is locally compact, $\mathbf{K}_{\m,v}^*$ and $\mathbf{K}_{\m,\geom,v}^*$ have non-zero Haar measure. Since the integrand in $\widehat{H}_{\m,v}(\delta_v, 1;-\mathbf{s})$ is non-negative for $\mathbf{s} \in \R_{>0}$, 
\begin{equation*}
\widehat{H}_{\m,v}^*(\delta_{\m,v}, 1;-\mathbf{s}) \geq \int_{\mathbf{K}_{\m,v}^*}\f{1}{H_{\m,v}(\mathbf{s},t_v)}\textrm{d}\mu_v >0,
\end{equation*}
and similarly for the geometric case.
\end{proof}

\subsection{Tauberian theorem}

Our ultimate aim is to apply Delange's Tauberian theorem to our height zeta functions. The version we give below is a specialisation of the one given in \cite[Thm.~3.3]{LOU}, which is based on the work of Delange \cite[Thm.~III]{DEL}.

\begin{proposition}[Tauberian theorem] \label{thm:taub}
Suppose that there exist $a,b \in \mathbb{R}_{>0}$ such that $Z_{\m,\geom}^*(s)$ is absolutely convergent for $\Re(s) >a$ and $f^*_{\mathbf{m},\geom}(s)= Z_{\m,\geom}^*(s)(s-a)^b$ can be extended to a holomorphic function on $\Re(s) \geq a$ which is non-zero at $s=a$ and satisfies
\begin{equation}
Z_{\mathbf{m},\geom}^*(s) = \frac{f^*_{\mathbf{m},\geom}(a)}{(s-a)^b} + O\l(\f{1}{(s-a)^{b-\delta}}\r) \textrm{ as }s \rightarrow a
\end{equation}
for some $\delta >0$. Then, for $\Gamma$ the gamma function,
\begin{equation}
N_{\geom}(\Sigma, \m,S,*;B) \sim \frac{f^*_{\mathbf{m},\geom}(a)}{\Gamma(b)}B^a(\log B)^{b-1}.
\end{equation}
\end{proposition}

\subsection{Poisson summation formula}

We will use the following form of the Poisson summation formula, which is a special case of \cite[Cor.~3.36]{BOU}.

\begin{proposition}[Poisson summation formula] \label{prop:psf}
Suppose that, for $\Re(\mathbf{s}) > 1$, the functions $P \mapsto \frac{\delta_{\m,\geom}^*(P)}{H_\m(\mathbf{s},P)}$ and $\chi \mapsto \widehat{H}_{\m}(\delta_{\m,\geom}^*, \chi; -\mathbf{s})$ are $L^1$ on $T(\mathbb{A}_K)$ and $T(\AA_K)/ \mathbf{K}_{\m,\geom}^* T(K)$ respectively. Then, in this region of $\mathbb{C}^r$, we have the equalities
\[
\begin{aligned}
Z_{\m}^*(\mathbf{s}) & = \f{1}{\left(2\pi\right)^{\rank X^*(T)}\vol(T(\AA_K)^1/T(K))}\int_{\chi \in (T(\AA_K)/ \mathbf{K}_{\m}^* T(K))^{\wedge}}\widehat{H}_{\m}(\delta_{\m}^*, \chi; -\mathbf{s})d\mu, \\
Z_{\m,\geom}^*(\mathbf{s}) & = \f{1}{\left(2\pi\right)^{\rank X^*(T)}\vol(T(\AA_K)^1/T(K))}\int_{\chi \in (T(\AA_K)/ \mathbf{K}_{\m,\geom}^* T(K))^{\wedge}}\widehat{H}_{\m}(\delta_{\m,\geom}^*, \chi; -\mathbf{s})d\mu.
\end{aligned}
\]
\end{proposition}

\subsection{Hecke characters}
\begin{mydef}
A \emph{Hecke character} over a number field $K$ is an automorphic character of $\mathbb{G}_{m,K}$, i.e.\ a continuous homomorphism $\chi = \left(\chi_v\right)_v: \mathbb{G}_m\left(\mathbb{A}_K\right) = \AA_K^* \rightarrow S^1$ such that $\prod_{v \in \Omega_K}\chi_v\left(a\right) = 1$ for all $a \in \mathbb{G}_m\left(K\right) = K^*$.
\end{mydef}

\begin{example}
As a first example of a non-trivial Hecke character, we have totally imaginary powers of the \emph{adelic absolute value map}
\[
\|\cdot\|_K: \mathbb{A}_K^* \rightarrow S^1, \quad \left(x_v\right)_v \mapsto \prod_{v \in \Omega_K}|x_v|_v.
\]
(We suppress ths subscript $K$ when it is clear from context.) That $\|\cdot\|^{i\theta}$, $\theta \in \mathbb{R}$ defines a Hecke character follows from Artin's product formula \cite[Prop.~III.1.3]{NEU}.
The kernel of $\|\cdot\|$ is denoted by $\mathbb{G}_m\left(\mathbb{A}_K\right)^1$.
\end{example}

\begin{mydef}
A Hecke character $\chi$ over $K$ is \emph{principal} if $\chi = \|\cdot\|^{i\theta}$ for some $\theta \in \mathbb{R}$.
\end{mydef}

\begin{mydef} \label{def:HeckeL}
The \emph{Hecke $L$-function} associated to a Hecke character $\chi$ over $K$ is the complex function given for $\Re(s) > 1$ by the Euler product
\[
L\left(\chi,s\right) = \prod_{v \in \Omega_K^f}L_v\left(\chi,s\right),
\]
where
$L_v\left(\chi,s\right) = 
\left(1-\frac{\chi_v\left(\pi_v\right)}{q_v^s}\right)^{-1}$ when $\chi$ is unramified at $v$ and is $1$ otherwise.
\end{mydef}

\begin{example}
For $\mathbf{1}$ the trivial Hecke character, we obtain the \emph{Dedekind zeta function}
\[
L\left(\mathbf{1},s\right) = \zeta_K\left(s\right) = \prod_{v \in \Omega_K^f}\left(1-\frac{1}{q_v^s}\right)^{-1}.
\]

More generally, for $\theta \in \mathbb{R}$, we have 
\[
L\left(\|\cdot\|^{i\theta},s\right) = \zeta_K\left(s+i\theta\right).
\]
\end{example}

\begin{mydef}
We call $\chi_\infty = (\chi_v)_{v \mid \infty}$ the \emph{infinity type} of a Hecke character $\chi$ over $K$.
\end{mydef}

For $v \mid \infty$, we have $\chi_v|_{\RR_{>0}} = |\cdot|_v^{i \kappa_v}$ for some $\kappa_v \in \mathbb{R}$, and we set $\|\chi_\infty\| = \max_{v \mid \infty}|\kappa_v|$.

The importance of Hecke $L$-functions for us is that they are relatively well understood and well behaved from an analytic perspective, as the following two results show.

\begin{proposition} \label{prop:cont} \cite[\S6]{HEC}
Let $\chi$ be a Hecke character over a number field $K$. Then $L\left(\chi,s\right)$ admits a meromorphic continuation to $\mathbb{C}$; this continuation has a simple pole at $s = 1 - i\theta$ if $\chi = \|\cdot\|^{i\theta}$ for some $\theta \in \mathbb{R}$, and is holomorphic if $\chi$ is non-principal.
\end{proposition}

\begin{proposition} \cite[Exercise~3,~\S5.2]{IK} \label{prop:hecbd}
Let $\chi$ be a non-principal Hecke character of $K$, $C$ be a compact subset of $\Re(s) \geq 1$ and $\varepsilon > 0$. Then, for $q(\chi)$ the conductor of $\chi$,
\[
L\left(\chi,s\right) \ll_{\varepsilon, C} q\left(\chi\right)^{\varepsilon} \left(1+ \|\chi_{\infty}\|\right)^{\varepsilon}, \quad \left(s-1\right)\zeta_K\left(s\right) \ll_C 1, \quad s \in C.
\]
\end{proposition}

\subsection{Character correspondence}

Let $X_{\Sigma}$ be a compactification of a torus $T/K$ with the extensions $K_i/K$, $i=1,\dots,r$ as before. To each automorphic character $\chi$ of $T$ we may associate Hecke characters $\chi_i$ over $K_i$, $i=1,\dots,r$ as in \cite[\S3.1]{BT}. Explicitly, there is a morphism $\gamma: \prod_{i=1}^r R_{K_i/K}\mathbb{G}_m \rightarrow T$. Thus, given a character $\chi \in \left(T(\mathbb{A}_K)/T(K)\right)^{\wedge}$, one obtains $r$ Hecke characters $\chi_i: \mathbb{A}_{K_i}^*/K_i^* \rightarrow S^1$. Of importance to us is the following fact.

\begin{lemma}
If $\chi_v$ is trivial on $\mathbf{K}_{\m,\geom,v}^*$, then, for each $i \in \{1,\dots,r\}$ and each $w \in \Omega_{K_i}$ over $v$, there exists a compact open subgroup $\mathbf{L}_{\m,\geom,w}^* \subset \calO_w^*$ of finite index on which the Hecke character $\chi_i$ is $1$. Moreover, when $v \not\in S(\Sigma,\calX)$, we have $\mathbf{L}_{\m,\geom,w}^* = \O_w^*$.
\end{lemma}

\section{Fan functions} \label{sec:fan}
In order to ``regularise'' (approximate) Fourier transforms in the height zeta function method, Batyrev and Tschinkel used the degree maps (Section \ref{sec:degmaps}) to realise local Fourier transforms as multi-dimensional geometric series defined via the fan $\Sigma$. They showed that these functions are well approximated by the local factors of certain Hecke $L$-functions.

In this section we recall the Batyrev--Tschinkel fan functions and develop analogues for semi-integral points by excising certain terms. Let $\Sigma$ be the fan of our toric variety, and let $\Sigma(1) = \bigcup_{i,w}\Sigma_{i,w}(1)$ denote the decomposition into $G_v$-orbits for a non-archimedean place $v$ of $K$. Recall (Proposition~\ref{prop:bij}) that the $\Sigma_{i,w}(1)$ are in bijection with places $w$ of $K_i = E^{G_i}$ over $v$, and that the length $f_{i,w}$ of $\Sigma_{i,w}(1)$ equals the inertia degree of $w$ over $v$. 

\begin{mydef}
Instantiate for each $(i,w)$ a variable $u_{i,w}$. For each $\sigma \in \Sigma^{G_v}$, set $I_v(\sigma) = \{(i,w): \Sigma_{i,w}(1) \subset \sigma(1) \}$. Let $\mathbf{u} = (u_{i,w})_{i,w}$. Define the functions $R_{\sigma,v}(\mathbf{u})$ and $Q_{\Sigma,v}(\mathbf{u})$ by
\[
\begin{aligned}
R_{\sigma,v}(\mathbf{u}) & := \prod_{(i,w) \in I_v(\sigma)}\frac{u_{i,w}^{f_{i,w}}}{1-u_{i,w}^{f_{i,w}}} & \in \QQ(\mathbf{u}), \\
\sum_{\sigma \in \Sigma^{G_v}}R_{\sigma,v}(\mathbf{u}) & = Q_{\Sigma,v}(\mathbf{u})\prod_{i,w}\left(1-u_{i,w}^{f_{i,w}}\right)^{-1}  & \in \QQ[\mathbf{u}].
\end{aligned}
\]
\end{mydef}

\begin{note}
Note that each factor of $R_{\sigma,v}(\mathbf{u})$ is a geometric series $\frac{x^f}{1-x^f} = \sum_{r=1}^\infty x^{rf}$.

Excising the terms $x^{rf}$, $r < m$ leaves
\[
\sum_{r = m}^\infty x^{rf} = \frac{x^{mf}}{1-x^f} = \frac{x^{mf}}{1-x^{mf}} + O\left(x^{(m+1)f}\right).
\]

Excising instead those terms $x^{rf}$ with $r \nmid m$ leaves
\[
\sum_{r = 1}^\infty x^{rmf} = \frac{x^{mf}}{1-x^{mf}}.
\]
\end{note}

\begin{mydef}
For $\mathbf{m} \in \ZZ_{\geq 1}^r$, define the functions $R^{\cc}_{\sigma,\mathbf{m},\geom,v}(\mathbf{u})$ and $Q^{\cc}_{\Sigma,\mathbf{m},\geom,v}(\mathbf{u})$ by
\[
\begin{aligned}
R^{\cc}_{\sigma,\mathbf{m},\geom,v}(\mathbf{u}) &:= \prod_{(i,w) \in I_v(\sigma)}\frac{u_{i,w}^{m_if_{i,w}}}{1-u_{i,w}^{f_{i,w}}} & \in \QQ(\mathbf{u}), \\
\sum_{\sigma \in \Sigma^{G_v}}R^{\cc}_{\sigma,\mathbf{m},\geom,v}(\mathbf{u}) &= Q^{\cc}_{\Sigma,\mathbf{m},\geom,v}(\mathbf{u})\prod_{i,w}\left(1-u_{i,w}^{m_if_{i,w}}\right)^{-1} & \in \QQ[\mathbf{u}].
\end{aligned}
\]
(Note the difference in denominators between the $R^{\cc}_{\sigma,\mathbf{m},\geom,v}(\mathbf{u})$ and the above sum.)

Define also the functions $R^{\dd}_{\sigma,\mathbf{m},\geom,v}(\mathbf{u})$ and $Q^{\dd}_{\Sigma,\mathbf{m},\geom,v}(\mathbf{u})$ by
\[
\begin{aligned}
R^{\dd}_{\sigma,\mathbf{m},\geom,v}(\mathbf{u}) & := \prod_{(i,w) \in I_v(\sigma)}\frac{u_{i,w}^{m_i f_{i,w}}}{1-u_{i,w}^{m_if_{i,w}}} & \in \QQ(\mathbf{u}), \\
\sum_{\sigma \in \Sigma^{G_v}}R^{\dd}_{\sigma,\mathbf{m},\geom,v}(\mathbf{u}) & = Q^{\dd}_{\Sigma,\mathbf{m},\geom,v}(\mathbf{u})\prod_{i,w}\left(1-u_{i,w}^{m_if_{i,w}}\right)^{-1} & \in \QQ[\mathbf{u}].
\end{aligned}
\]
\end{mydef}

We now prove that $Q^{*}_{\Sigma,\mathbf{m},\geom,v}(\mathbf{u}) - 1$ has high degree. Fixing $i = i_0$, the \emph{$i_0$-degree} $\deg_{i_0}(f)$ of $f \in \QQ[\mathbf{u}]$ is its degree with respect to the variables $u_{i_0,w}$. We set

\[
A_{\Sigma,v,i} :=\{\sigma \in \Sigma^{G_v}: I_v(\sigma) = \{(i,w_0)\}, f_{i,w_0} = 1\}, \quad
A_{\Sigma,v}:=\bigcup_{i=1}^rA_{\Sigma,v,i},
\]

\begin{proposition} \label{prop:fanfuncgeom}
\begin{enumerate} [leftmargin=*]
\item We have $Q^{\cc}_{\Sigma,\mathbf{m},\geom,v}(\mathbf{u}) = 1 + P^{\cc}_{\Sigma,\mathbf{m},\geom,v}(\mathbf{u})$, 
where $P^{\cc}_{\Sigma,\mathbf{m},\geom,v}(\mathbf{u}) \in \QQ[\mathbf{u}]$ satisfies $\deg_i(P^{\cc}_{\Sigma,\mathbf{m},\geom,v}(\mathbf{u})) \geq  m_i + 1$ for all $i \in \{1,\dots,r\}$.
\item Set $\mathbf{u}^{\mathbf{m}}:=\left(u_{i,w}^{m_i}\right)_{i,w}$. We have
$Q^{\dd}_{\Sigma,\mathbf{m},\geom,v}\left(\mathbf{u}\right) = 1 + P^{\dd}_{\Sigma,\geom,v}\left(\mathbf{u}^\mathbf{m}\right)$,
where $P^{\dd}_{\Sigma,\geom,v}(\mathbf{v}) \in \QQ[\mathbf{v}]$ satisfies $\deg(P^{\dd}_{\Sigma,\geom,v}(\mathbf{v})) \geq 2$.
\end{enumerate}
\end{proposition}

\begin{proof}
\begin{enumerate}
\item Fix $i$. Clearing denominators, $Q^{\cc}_{\Sigma,\mathbf{m},\geom,v}(\mathbf{u}) = \sum_{\sigma \in G_v} N^{\cc}_{\sigma,\mathbf{m},\geom,v}(\mathbf{u})$,
where
\[
N^{\cc}_{\sigma,\mathbf{m},\geom,v}(\mathbf{u}) = \prod_{(i,w) \in I_v(\sigma)}u_{i,w}^{m_if_{i,w}}\frac{1-u_i^{m_i f_{i,w}}}{1-u_i^{f_{i,w}}}\prod_{(i,w) \not\in I_v(\sigma)}\left(1-u_{i,w}^{m_if_{i,w}}\right) \in \QQ[\mathbf{u}].
\]

Using the identity $\frac{1-x^n}{1-x} = \sum_{i=0}^{n-1}x^i \in \QQ[x],$
we obtain 
\[
N^{\cc}_{\sigma,\mathbf{m},\geom,v}(\mathbf{u}) = \prod_{(i,w) \in I_v(\sigma)}\left(u_{i,w}^{m_if_{i,w}}\sum_{j=0}^{m_i - 1} u_i^{jf_{i,w}}\right) \prod_{(i,w) \not\in I_v(\sigma)}\left(1-u_{i,w}^{m_if_{i,w}}\right) \in \QQ[\mathbf{u}].
\]
Note that the $i$-degree of each monomial summand of $N^{\cc}_{\sigma,\mathbf{m},\geom,v}(\mathbf{u})$ is either $0$ or at least $m_if_{i,w}$ for some $w$; in particular, when it is positive, it is at least $m_i$, and is equal to $m_i$ if and only if either $\sigma = 0$ or $\sigma \in A_{\Sigma,v,i}$.

For $\sigma = 0$, there is $E^{\cc}_{0,\mathbf{m},\geom,v}(\mathbf{u})$ with $\deg_i(E^{\cc}_{0,\mathbf{m},\geom,v}(\mathbf{u})) \geq m_i +1$ such that
\[
N^{\cc}_{0,\mathbf{m},\geom,v}(\mathbf{u}) = \prod_{(i,w) \not\in I_v(0)}\left(1-u_{i,w}^{m_if_{i,w}}\right) = 1 - \sum_{\sigma' \in A_{\Sigma,v}} \prod_{(i,w) \in I_v(\sigma')}u_{i,w}^{m_i} + E^{\dd}_{0,\mathbf{m},\geom,v}(\mathbf{u}).
\]

For $\sigma \in A_{\Sigma,v,i}$, there is $E^{\cc}_{\sigma,\mathbf{m},\geom,v}(\mathbf{u})$ with $\deg_i(E^{\cc}_{\sigma,\mathbf{m},\geom,v}(\mathbf{u})) \geq m_i +1$ such that
\[
\begin{aligned}
N^{\cc}_{\sigma,\mathbf{m},\geom,v}(\mathbf{u}) & = \prod_{(i,w) \in I_v(\sigma)}\left(u_{i,w}^{m_i}\sum_{j=0}^{m_i-1}u_{i,w}^j\right) \prod_{(i,w) \neq (i_0,w_0)}\left(1-u_{i,w}^{m_if_{i,w}}\right) \\
& = \prod_{(i,w) \in I_v(\sigma)}u_{i,w}^{m_i} + E^{\cc}_{\sigma,\mathbf{m},\geom,v}(\mathbf{u}).
\end{aligned}
\]

Then $\deg_i\left(N^{\cc}_{0,\mathbf{m},\geom,v}(\mathbf{u}) + \sum_{\sigma \in A_{\Sigma,v,i}} N^{\cc}_{\sigma,\mathbf{m},\geom,v}(\mathbf{u})\right) \geq m_i + 1$, and we are done.

\item This follows from \cite[Prop.~2.2.3]{BT} upon introducing $v_{i,w}:=u_{i,w}^{m_i}$; we give the proof for clarity.

Set $v_{i,w} = u_{i,w}^{m_i}$ and $\mathbf{v} = (\mathbf{v}_{i,w})_{i,w}$. Clearing denominators, $Q^{\dd}_{\Sigma,\mathbf{m},\geom,v}(\mathbf{u}) = \sum_{\sigma \in G_v} N^{\dd}_{\sigma,\mathbf{m},\geom,v}(\mathbf{v})$,
where
\[
N^{\dd}_{\sigma,\mathbf{m},\geom,v}(\mathbf{v}) = \prod_{(i,w) \in I_v(\sigma)}v_{i,w}^{f_{i,w}}\prod_{(i,w) \not\in I_v(\sigma)}\left(1-v_{i,w}^{f_{i,w}}\right) \in \QQ[\mathbf{u}].
\]
Now, $\deg\left(N^{\dd}_{\sigma,\mathbf{m},\geom,v}\right) = \sum_{(i,w) \in I_v(\sigma)}f_{i,w}$, so $\deg\left(N^{\dd}_{\sigma,\mathbf{m},\geom,v}\right) \leq 1$ iff $\sigma \in A_{\Sigma,v} \cup \{0\}$. Thus, it suffices to show that $\deg\left(N^{\dd}_{0,\mathbf{m},\geom,v}(\mathbf{v}) + \sum_{\sigma \in A_{\Sigma,v}} N^{\dd}_{\sigma,\mathbf{m},\geom,v}(\mathbf{v})\right) \geq 2$.

For $\sigma = 0$, there is $E^{\dd}_{0,\mathbf{m},\geom,v}(\mathbf{v})$ with $\deg\left(E^{\dd}_{0,\mathbf{m},\geom,v}(\mathbf{v})\right) \geq 2$ such that
\[
N^{\dd}_{0,\mathbf{m},\geom,v}(\mathbf{v}) = \prod_{(i,w) \not\in I_v(0)}\left(1-v_{i,w}^{f_{i,w}}\right) = 1 - \sum_{\sigma' \in A_{\Sigma,v}} \prod_{(i,w) \in I_v(\sigma')}v_{i,w} + E^{\dd}_{0,\mathbf{m},\geom,v}(\mathbf{v}).
\]

For $\sigma \in A_{\Sigma,v}$, there is $E^{\dd}_{\sigma,\mathbf{m},\geom,v}(\mathbf{v})$ with $\deg\left(E^{\dd}_{\sigma,\mathbf{m},\geom,v}(\mathbf{v})\right) \geq 2$ such that
\[
N^{\dd}_{\sigma,\mathbf{m},\geom,v}(\mathbf{v}) = \prod_{(i,w) \in I_v(\sigma)}v_{i,w} \prod_{(i,w) \neq (i_0,w_0)}\left(1-v_{i,w}^{f_{i,w}}\right) = \prod_{(i,w) \in I_v(\sigma)}v_{i,w} + E^{\dd}_{\sigma,\mathbf{m},\geom,v}(\mathbf{v}).
\]

Then $\deg\left(N^{\dd}_{0,\mathbf{m},\geom,v}(\mathbf{v}) + \sum_{\sigma \in A_{\Sigma,v}} N^{\dd}_{\sigma,\mathbf{m},\geom,v}(\mathbf{v})\right) \geq 2$, and we are done. \qedhere
\end{enumerate}
\end{proof}

\section{Proof of main results} \label{sec:geomproof}

In this section we prove Theorem~\ref{thm:geom} and Theorem~\ref{thm:camp}, which verifies the modified PSTVA conjecture for smooth toric orbifolds. We also deduce Corollary~\ref{cor:split}.

\subsection{Regularisation}

\begin{proposition} \label{prop:geomreg}
For $\chi \in \left(T(\AA_K)/T(K)\right)^\wedge$ with $\chi(\mathbf{K}_{\m,\geom}^*) = 1$ and all but finitely many places $v$,
\[
\widehat{H}_{\m,v}(\delta_{\m,\geom,v}^*, \chi_v;-\mathbf{s}) = \left(\prod_{i=1}^r \prod_{\substack{w \in \Omega_{K_i} \\ w \mid v}}L_w(\chi_i^{m_i},s_i)\right)Q^*_{\Sigma,\mathbf{m},\geom,v}\left(\frac{\chi_v(\mathbf{n_{i,w}})}{q_v^{\mathbf{s}/\mathbf{m}}}\right).
\]
\end{proposition}

\begin{proof}
We adapt the proof of \cite[Thm.~2.2.6]{BT}. For all but finitely many $v$, we have $\mu_v(T(\O_v)) = 1$ and $T(K_v)/T(\mathcal{O}_v) \cong X_*(T_v)$, thus we may view $H_{\m,v}(\mathbf{s},\cdot)$, $\delta_{\m,\strong,v}^*$ and $\chi_v$ as functions on $X_*(T_v)$, and we have 
\[
\widehat{H}_{\m,v}(\delta_{\m,\geom,v}^*, \chi_v;-\mathbf{s}) = \sum_{n_v \in X_*(T_v)}\delta_{\m,\geom,v}^*(n_v)\frac{\chi_v(n_v)}{H_{\m,v}(\mathbf{s},n_v)}.
\]

We partition the space $X_*(T_v)$ into the relative interiors of cones $\sigma \in \Sigma^{G_v}$, obtaining
\[
\widehat{H}_{\m,v}(\delta_{\m,\geom,v}^*, \chi_v;-\mathbf{s}) = \sum_{\sigma \in \Sigma^{G_v}}\left(\sum_{n_v \in X_*(T_v) \cap \sigma^{\circ}}\delta_{\m,\geom,v}^*(n_v)\frac{\chi_v(n_v)}{H_{\m,v}(\mathbf{s},n_v)}\right).
\]

For $\sigma \in \Sigma^{G_v}$ and $n_v \in \sigma$, we may write $n_v = \sum_{(i,w) \in I_v(\sigma)}\alpha_{i,w}\overline{e}_{i,w}$ where $\alpha_{i,w} \in \ZZ_{> 0}$. By Corollary \ref{cor:maincor}, $\delta_{\m,\geom,v}^*(n_v) = 1$ iff $\alpha_{i,w} \in  \mathbb{Z}_{\geq m_i} \cup \{0\}$ for all $(i,w) \in I_v(\sigma)$ (when $* = \cc$) or $m_i \mid \alpha_{i,w}$ for all $(i,w) \in I_v(\sigma)$ (when $* = \dd$).

Pick a representative $n_{i,w}$ for each $\Sigma_{i,w}(1)$. Then

\[
\sum_{n_v \in X_*(T_v) \cap \sigma^{\circ}}\delta_{\m,\geom,v}^*(n_v)\frac{\chi_v(n_v)}{H_{\m,v}(\mathbf{s},n_v)} = R^*_{\sigma,\mathbf{m},\geom,v}\left(\frac{\chi_v(\mathbf{n}_{i,w})}{q_v^{s_i/m_i}}\right),
\]
and the result follows from the definition of $Q^*_{\Sigma,\mathbf{m},\geom,v}$.
\end{proof}

\begin{corollary} \label{cor:campreg}
For $\chi$ as in Proposition~\ref{prop:geomreg}, we have
\[
\widehat{H}_{\m}\left(\delta^*_{\m,\geom},\chi;-\mathbf{s}\right) = \prod_{i=1}^r L(\chi_i^{m_i},s_i)G^*_{\m,\geom}(\chi,\mathbf{s}),
\]
where $G^*_{\m,\geom}(\chi,\mathbf{s})$ is holomorphic and uniformly bounded with respect to $\chi$ on $\Re(\mathbf{s}) \geq \frac{\mathbf{m}}{\mathbf{m}+\mathbf{1}}$ and $G^*_{\m,\geom}(1,\mathbf{1}) \neq 0$.
\end{corollary}

\begin{proof}
The result now follows from Propositions~\ref{prop:geomreg}, \ref{prop:fanfuncgeom}, \ref{prop:hecbd} and \ref{prop:nonvan}.
\end{proof}

\begin{corollary}
If $(X_{\Sigma},D_{\Sigma,\m})$ is smooth, then, for $\chi$ as in Proposition~\ref{prop:geomreg}, we have 
\[
\widehat{H}_{\m}\left(\delta^*_{\m},\chi;-\mathbf{s}\right) = \prod_{i=1}^r L(\chi_i^{m_i},s_i)G^*_{\m}(\chi,\mathbf{s}),
\]
where $G^*_{\m}(\chi,\mathbf{s})$ is holomorphic and uniformly bounded with respect to $\chi$ on $\Re(\mathbf{s}) \geq \frac{\mathbf{m}}{\mathbf{m}+\mathbf{1}}$ and $G^*_{\m}(1,\mathbf{1}) \neq 0$.
\end{corollary}

\begin{proof}
For all but finitely many $v$, we have $\widehat{H}_{\m,v}\left(\delta^*_{\m,\geom,v},\chi_v;-\mathbf{s}\right) = \widehat{H}_{\m}\left(\delta^*_{\m,v},\chi_v;-\mathbf{s}\right)$, thus the result follows from Propositions~\ref{prop:geomreg}, \ref{prop:fanfuncgeom} and \ref{prop:nonvan}.
\end{proof}

\subsection{Poisson summation formula}

Henceforth, set $t:= \rank X^*(T)$.

\begin{proposition}
For $\Re(\mathbf{s}) > 1$, we have
\begin{equation} \label{eq:psf2}
Z_{\m,\geom}^*(\mathbf{s}) = \f{1}{\left(2\pi\right)^t\vol(T(\AA_K)^1/T(K))}\int_{\chi \in (T(\AA_K)/ \mathbf{K}_{\m,\geom}^* \cdot T(K))^{\wedge}}\widehat{H}_{\m}(\delta_{\m,\geom}^*, \chi; -\mathbf{s})d\mu.
\end{equation} 
\end{proposition}

\begin{proof}
By Proposition~\ref{prop:psf}, it suffices to show that: 
\begin{enumerate}
    \item the function $P \mapsto \frac{\delta_{\m,\geom}^*(P)}{H_\m(\mathbf{s},P)}$ is $L^1$ on $T(\A_K)$, and
    \item the integral on the right-hand side of (\ref{eq:psf2}) is absolutely convergent.
\end{enumerate}
The first requirement follows directly from Proposition~\ref{prop:geomreg}.

The second requirement follows from a technical result of Batyrev and Tschinkel \cite[Cor.~4.6]{BT2}, which makes use only of the fact that the regularisations are products of $L$-functions, thus uniformly bounded in compact subsets of $\Re(\mathbf{s}) > 1$; Proposition~\ref{prop:geomreg} and Corollary~\ref{cor:campreg} tell us that the same is true here.
\end{proof}

We are now ready to prove Theorem~\ref{thm:geom} and deduce Theorem~\ref{thm:camp} and Corollary~\ref{cor:split}.

\begin{proof}[Proof of Theorem~\ref{thm:geom}]
By the product formula, $T(K) \subset T(\mathbb{A}_K)^1$ \cite[p.~28]{BOU}, thus we have a morphism $T(\A_K)/T(K) \rightarrow T(\A_K)/T(\A_K)^1$. Non-canonically splitting
\[
0 \rightarrow T(\A_K)^1 \rightarrow T(\A_K) \rightarrow T(\A_K)/T(\A_K)^1 \rightarrow 0,
\]
we thus obtain a splitting $T(\A_K)/T(\A_K)^1 \rightarrow T(\A_K)/T(K)$.

This splitting gives a splitting of the short exact sequence
\[
0 \rightarrow T(\A_K)^1/T(K) \rightarrow T(\A_K)/T(K) \rightarrow T(\A_K)/T(\A_K)^1 \rightarrow 0,
\]
and so we obtain a non-canonical splitting of automorphic characters
\[
(T(\A_K)/T(K))^\wedge \xrightarrow{\sim} (T(\A_K)/T(\A_K)^1)^\wedge \times (T(\A_K)^1/T(K))^\wedge, \quad \chi \mapsto (\chi_y,\chi_l). 
\]
Following \cite[p.~109]{BOU}, denote by $\widetilde{\calU_T}$ the subgroup of $(T(\A_K)/T(K))^\wedge$ identified with the factor $(T(\A_K)^1/T(K))^\wedge$, and denote by $\calU_{T,\m}^*$ (respectively, $\calU_{T,\m,\geom}^*$) the image of $\left(T(\A_K)/\mathbf{K}_{\m}^* T(K)\right)^\wedge$ (respectively, $\left(T(\A_K)/\mathbf{K}_{\m,\geom}^* T(K)\right)^\wedge$) in $\widetilde{\calU_T}$ via this isomorphism. 

Using the splitting of characters discussed above, the isomorphism $X^*(T)_\mathbb{R} \cong \mathbb{R}^t$ and the correct Haar measures throughout, \eqref{eq:psf2} may be further re-expressed as

\begin{equation} \label{eq:psf3}
Z_{\m,\geom}^*(\mathbf{s}) = \frac{1}{(2\pi)^t\vol(T(\A_K)^1/T(K))}\int_{\mathbf{y} \in \R^t}\left( \sum_{\chi \in \calU_{T,\m,\geom}^*}\widehat{H}(\delta_{\m,\geom}^*, \chi; -\mathbf{m}^{-1} \mathbf{s} - i \gamma_\R(\mathbf{y}))\right)d\mathbf{y}.
\end{equation} 

Integrals of the form \eqref{eq:psf3} are the subject of a technical result in Bourqui's exposition \cite[Thm.~5.7]{BOU} due originally to Chambert-Loir and Tschinkel; an application of this theorem to \eqref{eq:psf3} means that $Z_{\m,\geom}^*(s)(s-1)^{\rank(\Pic(X_{\Sigma}))}$ can be extended to a holomorphic function on $\Re(s) \geq 1$ with value at $s = 1$ equal to
\[\frac{C\left|H^1(G,X^*(\overline{T}))\right|}{\vol(T(\A_K)^1/T(K))}\alpha(X_{\Sigma},-K_{X_\Sigma}), \quad
C := \lim_{s \rightarrow 1}(s-1)^{\rank(\Pic(X_{\Sigma}))}\sum_{\chi \in \calU_{T,\m,\geom}^*}\widehat{H}_{\m}(\delta_{\m,\geom}^*,\chi;-s).
\]
To conclude the proof via the Tauberian theorem (Theorem~\ref{thm:taub}), it suffices to show that this application is valid and to verify that $C$ is non-zero and agrees with Conjecture~\ref{conj:CLTBT}.

That the application of \cite[Thm.~5.7]{BOU} is valid follows directly from the fact that, as in the rational points case, the regularisation is a product of $L$-functions. 

In order to calculate $C$, we need only consider the subsum over \emph{contributing characters}, i.e.\ those characters for which every associated Hecke character $\chi_i$ is $m_i$-torsion, so that the associated regularisation is composed entirely of Dedekind zeta functions, i.e.\
\[
\calV_{T,\m,\geom}^* := \{\chi \in \calU_{T,\m,\geom}^*: \chi_i^{m_i} =1 \text{ for all $i=1,\dots,r$}\} \subset \calU_{T,\m,\geom}^*.
\]

When $\m = \mathbf{1}$ (the rational points case), we have$\left|\calV^*_{T,\mathbf{1},\geom}\right| = |A(T)| = \frac{\beta(X_{\Sigma})}{i(T)}$ for $A(T) = T(\AA_K)/\overline{T(K)}$, which is used in the verification of the conjectural constant. Generalising, we deduce from Lemma~\ref{lem:br} and the isomorphism $\Sh(T) \cong \B(T)^\sim$ \cite[\S4.2.3]{LOU} that
\[
\left|\calV_{T,\m,\geom}^*\right| = \frac{\left|\Br_1(X_{\Sigma},D_{\Sigma,\m})^{\mathbf{K}_{\m,\geom}^*}/\Br(K)\right|}{|\Sh(T)|}.
\]
The finiteness of the number of contributing characters, thus this quantity, is ensured by an appeal to global class field theory as in \cite[Proof~of~Lem.~5.20]{STR}. We have
\[
\sum_{\chi \in \calV_{T,\m,\geom}^*}\widehat{H}_{\m}(\delta^*_{\m,\geom},\chi;-\mathbf{s}) = \sum_{\chi \in \calV_{T,\m,\geom}^*} \int_{T(\mathbb{A}_K)}\frac{\delta^*_{\m,\geom}(t)\chi(t)}{H_{\m}(\mathbf{s},t)}d\mu = \int_{T(\mathbb{A}_K)}\frac{\delta^*_{\m,\geom}(t)}{H_{\m}(\mathbf{s},t)}\sum_{\chi \in \calV_{T,\m,\geom}^*}\chi(t)d\mu.
\]
By character orthogonality, this expression is equal to the integral
\begin{equation} \label{eq:contint}
|\calV_{T,\m,\geom}^*|\int_{T(\mathbb{A}_K)^*_{\m,\geom}}\frac{1}{H_{\m}(\mathbf{s},t)}d\mu,
\end{equation}
where
\[
T(\mathbb{A}_K)_{\m,\geom}^*:= \{t \in T(\mathbb{A}_K): \delta^*_{\m,\geom}(t) = \chi(t) = 1 \text{ for all $\chi \in \calU_{T,\m,\geom}^*$}\}.
\]

By Lemma~\ref{lem:br}, we deduce that 
\[
T(\mathbb{A}_K)_{\m,\geom}^* = T(\AA_K)^{\Br_1(X_{\Sigma},D_{\Sigma,\mathbf{m}})^{\mathbf{K}_{\m,\geom}^*}},
\]
where $\Br_1(X_{\Sigma},D_{\Sigma,\mathbf{m}})^{\mathbf{K}_{\m,
\geom}^*} = \{\calA \in \Br_1(X_{\Sigma},D_{\Sigma,\mathbf{m}}): \chi_{\calA} \in \calU_{T,\m,\geom}^*\}$.

Using $\left|H^1(G,X^*(\overline{T}))\right|L(X^*(\overline{T}),1) = \vol\left(T(\AA_K)^1/T(K)\right)\left|\Sh(T)\right|$ \cite[\S3.5]{ONO} and reinterpreting the conjectural leading constant as in \cite[Thm.~8.6]{CLTBT}, it suffices to show
\[
\begin{aligned}
&\lim_{\mathbf{s} \rightarrow \mathbf{1}}(\mathbf{s} - \mathbf{1})^{\rank(\Pic(X_{\Sigma}))}\int_{T(\AA_K)^{\Br_1(X_{\Sigma},D_{\Sigma,\mathbf{m}})^{\mathbf{K}_{\m,\geom}^*}}}\frac{1}{H_{\m}(\mathbf{s},t)}d\mu \\
= &L^*(X^*(\overline{T}),1)\lim_{S'}\tau_{X_{\Sigma},D_{\Sigma,\m},S'}\left(\left(\prod_{v \in S'}(X_{\Sigma},D_{\Sigma,\m})_{\geom}^*(\O_v) \cap T(K_v)\right)^{\Br_1(X_{\Sigma},D_{\Sigma,\m})^{\mathbf{K}^*_{\m,\geom}}}\right).
\end{aligned}
\]

The factor $L^*(X^*(\overline{T}),1)$ arises from the convergence factors of $d\mu$, and the remaining Tamagawa piece arises as in \cite[Thm.~8.6]{CLTBT}, with the sole difference that $d\mu$ is not normalised as the Haar measure is in \emph{loc.\ cit.}

The positivity of the integral in \eqref{eq:contint}, thus of the leading constant, can be verified as in \cite[Prop.~5.22]{STR}. We sketch the argument: introduce the refined indicator functions $\theta_{\m,\geom,v}: T(K_v) \rightarrow \{0,1\}$, $v \not\in S$ such that $\theta_{\m,\geom,v}(P) = 1$ iff $n_w(\overline{\calT}_j,P) \in \{0,m_i\}$ for each $w \mid v \in \Omega_E$, $j \in \{1,\dots,r\}$ and $\overline{T}_j$ a component of $D_i$. Denote the induced adelic indicator function by $\theta_{\m,\geom}$. Then $T(\AA_K)_{\m,\geom}^* \supset T(\AA_K)^{\theta_{\m,\geom}}:= \{(t_v)_v \in T(\AA_K): \theta_{\m,\geom}((t_v)_v) = 1\}$. Then it suffices to note that $\lim_{\mathbf{s} \rightarrow \mathbf{1}}\widehat{H}_{\m}(\theta_{\m,\geom},1;-\mathbf{s}) > 0$ (the limit being taken along the real line in each copy of $\CC$) as, mimicking the proof of Proposition~\ref{prop:geomreg}, we have
\[
\widehat{H}_{\m}(\theta_{\m,\geom},1;-\mathbf{s}) = \prod_{i=1}^r \zeta_{K_i}(s_i) \widetilde{G}^*_{\m,\geom}(\mathbf{s})
\]
for $\widetilde{G}^*_{\m,\geom}(\mathbf{s})$ a function holomorphic on $\Re(\mathbf{s}) \geq \frac{\m}{\m + \mathbf{1}}$ with $\widetilde{G}^*_{\m,\geom}(\mathbf{1}) \neq 0$.
\end{proof}

\begin{proof}[Proof of Theorem~\ref{thm:camp}]
The proof is almost identical to that of Theorem~\ref{thm:geom}: we replace the appeal to Proposition~\ref{prop:geomreg} by one to Corollary~\ref{cor:campreg}, and for the remainder of the argument, we replace all functions and groups by their non-geometric counterparts.
\end{proof}

\begin{proof}[Proof of Corollary~\ref{cor:split}]
By Lemma~\ref{lem:snc}, the orbifold $(X_\Sigma,D_{\Sigma,\m})$ is smooth when $X_{\Sigma}$ is split. The result then follows from Theorem~\ref{thm:camp}.
\end{proof}
\bibliographystyle{halpha-abbrv}
\bibliography{refs}
\end{document}